\documentclass[10pt,a4paper,twoside]{amsart}
\usepackage{amsmath,amssymb,times,textcomp,amsthm}
\usepackage[utf8]{inputenc}
\usepackage[T1]{fontenc}
\usepackage[all]{xy}
\usepackage[pdftex]{graphicx}

\newtheoremstyle{ourthm}
{3pt}
{3pt}
{\slshape}
{}
{\bf\upshape}
{}
{.5em}
{}

\theoremstyle{ourthm}
\newtheorem{theorem}{Theorem}[section]
\newtheorem{thmintro}{Theorem}
\newtheorem{lemma}[theorem]{Lemma}
\newtheorem{proposition}[theorem]{Proposition}
\newtheorem{corollary}[theorem]{Corollary}

\theoremstyle{definition}
\newtheorem{example}[theorem]{Example}

\newenvironment{proofof}[1]{\noindent{\sl Proof of #1.}}{\qed}
\newenvironment{beweis*}[1]{\noindent{\it #1}}{\qed}

\numberwithin{equation}{section}
\newcommand{\Section}[1]{\renewcommand{\thesection}{\S \arabic{section}}
\section{#1}\renewcommand{\thesection}{\arabic{section}}}

\newcommand{\ka}{{\mathcal A}}
\newcommand{\kb}{{\mathcal B}}
\newcommand{\kc}{{\mathcal C}}

\newcommand{\kf}{{\mathcal F}}

\newcommand{\kh}{{\mathcal H}}

\newcommand{\kn}{{\mathcal N}}
\newcommand{\ko}{{\mathcal O}}

\newcommand{\kr}{{\mathcal R}}

\newcommand{\ku}{{\mathcal U}}

\newcommand{\kw}{{\mathcal W}}

\newcommand{\ky}{{\mathcal Y}}
\newcommand{\kz}{{\mathcal Z}}

\newcommand{\IB}{{\mathbb B}}
\newcommand{\IC}{{\mathbb C}}

\newcommand{\IF}{{\mathbb F}}
\newcommand{\IG}{{\mathbb G}} 
\newcommand{\IH}{{\mathbb H}}

\newcommand{\JJ}{{\mathbb J}}

\newcommand{\IP}{{\mathbb P}} 
\newcommand{\IQ}{{\mathbb Q}} 

\newcommand{\IS}{{\mathbb S}}

\newcommand{\IX}{{\mathbb X}}
 
\newcommand{\IZ}{{\mathbb Z}}

\newcommand{\gothg}{{\mathfrak g}}

\newcommand{\gothl}{{\mathfrak l}}
\newcommand{\gothm}{{\mathfrak m}}

\newcommand{\beeq}[1]{\begin{eqnarray}\label{#1}}
\newcommand{\eneq}{\end{eqnarray}}

\newcommand{\isom}{\cong}

\newcommand{\tensor}{\otimes}

\newcommand{\epimorph}{\twoheadrightarrow}

\newcommand{\Bl}{{\rm Bl}}

\newcommand{\Cok}{{\rm coker}}

\newcommand{\End}{{\rm End}}

\newcommand{\Hom}{{\rm Hom}}
\newcommand{\khom}{\kh om}
\newcommand{\id}{{\rm id}}
\newcommand{\im}{{\rm im}}

\newcommand{\rk}{{\rm rk}}
\DeclareMathOperator{\Lie}{Lie}
\DeclareMathOperator{\LieGl}{GL}

\DeclareMathOperator{\LiePGl}{PGL}

\DeclareMathOperator{\Liegl}{\gothg\gothl}

\DeclareMathOperator{\td}{td}

\newcommand{\Kern}{{\rm ker}}

\newcommand{\Hilb}{{\rm Hilb}}
\newcommand{\Grass}{{\rm Grass}}

\newcommand{\Pic}{{\rm Pic}}

\newcommand{\Sing}{{\rm Sing}}
\renewcommand{\det}{{\rm det}}

\newcommand{\pr}{{\rm pr}}

\newcommand{\lra}{\longrightarrow}
\newcommand{\xra}{\xrightarrow}
\newcommand{\xla}{\xleftarrow} 
\newcommand{\lla}{\longleftarrow}

\newcommand{\verylongarrow}[1]{\raisebox{0mm}[0.9ex][0ex]{\hbox to #1{\rightarrowfill}}}

\DeclareMathOperator{\adj}{adj}

\DeclareMathOperator{\rad}{rad}
\DeclareMathOperator{\res}{res}
\DeclareMathOperator{\red}{red}

\newcommand{\sfrac}[2]{{\textstyle\frac{ #1}{ #2}}}

\DeclareMathOperator{\Sym}{Sym}
\DeclareMathOperator{\tr}{tr}
\DeclareMathOperator{\GIT}{/\!\!/}

\DeclareMathOperator{\ganz}{int}

\DeclareMathOperator{\PD}{PD}

\newcommand{\BB}{}

\begin{document}

\title{Twisted cubics on cubic fourfolds}
\author{Ch.\ Lehn, M.\ Lehn, Ch.\ Sorger, D. van Straten}

\address{\noindent Christian Lehn\\
Institut de Recherche Math\'ematique Avanc\'ee de Strasbourg\\
Universit\'e de Strasbourg\\
7, rue Ren\'e-Descartes\\ 
F-67084 Strasbourg, France} 
\email{lehn@math.unistra.fr}

\address{Manfred Lehn\\
Institut f\"ur Mathematik\\
Johannes Gutenberg--Universit\"at Mainz\\
D-55099 Mainz, Germany}
\email{lehn@mathematik.uni-mainz.de}

\address{Christoph Sorger\\
Laboratoire de Math\'ematiques Jean Leray\\
Université de Nantes\\
2, Rue de la Houssinière\\
BP 92208\\
F-44322 Nantes Cedex 03, France}
\email{christoph.sorger@univ-nantes.fr} 

\address{Duco van Straten\\
Institut f\"ur Mathematik\\
Johannes Gutenberg--Universit\"at Mainz\\
D-55099 Mainz, Germany}
\email{straten@mathematik.uni-mainz.de}


\begin{abstract} We construct a new twenty-dimensional family of projective 
eight-di\-men\-sio\-nal irreducible holomorphic symplectic manifolds: the compactified 
moduli space $M_3(Y)$ of twisted cubics on a smooth cubic fourfold $Y$ that 
does not contain a plane is shown to be smooth and to admit a contraction 
$M_3(Y)\to Z(Y)$ to a projective eight-dimensional symplectic manifold $Z(Y)$.
The construction is based on results on linear determinantal representations
of singular cubic surfaces.
\end{abstract}
\maketitle

\section*{Introduction}

According to Beauville and Donagi \cite{BeauvilleDonagi}, the Fano variety
$M_1(Y)$ of lines on a smooth cubic fourfold $Y\subset\IP^5_\IC$ 
is a smooth four-dimensional holomorphically symplectic variety which is
deformation equivalent to the second Hilbert scheme of a K3-surface.
The symplectic structure can be constructed as follows: let
$C\subset M_1(Y)\times Y$
denote the universal family of lines and let $\pr_i$ be the projection onto the
$i$-th factor of the ambient space. For any generator
$\alpha\in H^{3,1}(Y)\isom\IC$ one gets a holomorphic two-form
$\omega_1:=\pr_{1*}\pr_2^*\alpha$ on $M_1(Y)$.

More generally, one \BB may consider moduli spaces of smooth rational
curves of arbitrary degree $d$ on $Y$. For $d\geq 2$ such spaces are no longer
compact, and depending on the purpose one might consider compactifications in
the Chow variety or the Hilbert scheme of $Y$ or in the moduli space of stable
maps to $Y$. To be specific we let $M_d(Y)$ denote the compactification in the
Hilbert scheme $\Hilb^{dn+1}(Y)$. The moduli spaces $M_d(Y)$ and their rationality
properties have been studied by de Jong and Starr \cite{deJongStarr}. They showed
that any desingularisation of $M_d(Y)$ carries a canonical $2$-form $\omega_d$ which
at a generic point of $M_d(Y)$ is non-degenerate if $d$ is odd and $\geq 5$ and has
$1$-dimensional radical if $d$ is even and $\geq 6$. For the remaining small
values of $d$, de Jong and Starr found that the radical of the form has dimension
$3$, $2$ and $3$ at a generic point if $d=2$, $3$ or $4$, respectively.

The geometric reason for the degeneration of $\omega_2$ can be \BB seen as follows: 
Any rational curve $C$ of degree $2$ on $Y$ spans a two dimensional linear space 
$E\subset\IP^5$ which in turn cuts out a plane curve of degree $3$ from $Y$. As 
this curve contains $C$, it must have a line $L$ as residual component. Mapping 
$[C]$ to $[L]$ defines a natural morphism $M_2(Y)\to M_1(Y)$, the fibre over a point 
$[L]\in M_1(Y)$ being isomorphic to the three dimensional space of planes in $\IP^5$ 
that contain the line $L$. 

The geometry of $M_3(Y)$ is much more interesting. We show first:

\begin{thmintro} \label{thm:SmoothAndIrreducible} --- Let $Y\subset\IP^5$ be a 
smooth cubic hypersurface that does not contain a plane. Then the moduli space 
$M_3(Y)$ of generalised twisted cubic curves on $Y$ is a smooth and irreducible 
projective variety of dimension 10. 
\end{thmintro}

Let $\omega_3$ denote the holomorphic $2$-form defined by de Jong and
Starr. The purpose of this paper is to produce a contraction $M_3(Y)\to Z$ to an 
$8$-dimensional symplectic manifold $Z$. More precisely, we will prove: 

\begin{thmintro} \label{thm:MainTheorem} --- Let $Y\subset\IP^5$ be a smooth cubic 
hypersurface that does not contain a plane. Then there is a smooth eight dimensional 
holomorphically symplectic manifold $Z$ and morphisms $u:M_3(Y)\to Z$ and $j:Y\to Z$ 
with the following properties:
\begin{enumerate}
\item The symplectic structure $\omega$ on $Z$ satisfies $u^*\omega=\omega_3$.
\item The morphism $j$ is a closed embedding of $Y$ as a Lagrangian submanifold in $Z$.
\item The morphism $u$ factors as follows:
$$
\xymatrix{
M_3(Y)\ar[rr]^{u}\ar[rd]_{a}&&Z\\
&Z',\ar[ur]_{\sigma}\\}
$$
where $a:M_3(Y)\to Z'$ is a $\IP^2$-fibre bundle and $\sigma:Z'\to Z$ is
the blow-up of $Z$ along $Y$.
\item The topological Euler number of $Z$ is $e(Z)=25650$.
\end{enumerate}
Moreover, $Z$ is simply connected, and $H^0(Z,\Omega^2_Z)=\IC\omega$. In particular,
$Z$ is an irreducible holomorphic symplectic manifold and carries a hyperkähler metric.
\end{thmintro}

Since $25650$ is also the Euler number of $\Hilb^4(K3)$, it seems likely that $Z$ is 
deformation equivalent to the fourth Hilbert scheme of a K3 surface. 

The manifold $Z$ does of course depend on $Y$ and should systematically be
denoted by $Z(Y)$. In order to increase the readability of the paper we have
decided to stick with $Z$. Nevertheless, the construction works well for any
flat family $\ky\to T$ of smooth cubic fourfolds without planes and yields
a family $\kz\to T$ of symplectic manifolds.

The two-step contraction $u:M_3(Y)\to Z$ has an interesting interpretation
in terms of matrix factorisations. Let $P=\IC[x_0,\ldots,x_5]$ and let $R=P/f$, 
where $f$ is the equation of
a smooth cubic hypersurface $Y\subset \IP^5$. The ideal $I\subset R$ of
a generalised twisted cubic $C\subset Y$ is generated by two linear forms
and three quadratic forms. As Eisenbud \cite{Eisenbud} has shown, the
minimal free resolution
$$0\lla I\lla R_0\lla R_1\lla R_2\lla\ldots$$
becomes 2-periodic for an appropriate choice of bases for the free $R$-modules
$R_i$. Going back in the resolution, information about $I$ gets lost at each step
before stabilisation sets in. One can show that this stepwise loss of information
corresponds exactly to the two phases 
$$M_3(Y) \to Z'\quad\text{and}\quad Z'\to Z$$
of the contraction of $M_3(Y)$. Thus periodicity begins one step earlier for
curves that are arithmetically Cohen-Macaulay (aCM) than for those that are 
not (non-CM). Consequently, $Z$ truly parameterises
isomorphism classes of Cohen-Macaulay approximations in the sense of 
Auslander and Buchweitz \cite{AuslanderBuchweitz}. We intend to return to these
questions in a subsequent paper.\BB\\

{\noindent\bf Structure of the paper.} In Section \ref{sec:HilbertSchemes} we introduce
the basic objects of the discussion: generalised twisted cubics and their
moduli space. The focus lies on describing the possible degenerations
of a smooth twisted cubic space curve and understanding the fundamental difference between
curves that are arithmetically CM and those that are not. 
Any generalised twisted cubic $C$ spans a 3-dimensional projective space $\langle C\rangle$
and defines a cubic surface $S=Y\cap\langle C\rangle$. In Section 
\ref{sec:TwistedCubicsOnCubicSurfaces} we describe the moduli spaces of
generalised twisted cubics on possibly singular cubic surfaces $S$. Such
curves are related to linear determinantal representations of $S$.
In Section \ref{sec:ModuliOfDeterminantalRepresentations} we study this relation 
in the universal situation of integral cubic surfaces in a fixed $\IP^3$. This
is the technical heart of the paper. The main tool are methods from 
geometric invariant theory. The results obtained in this section will be applied 
in Section \ref{sec:NowOnY} to the family of cubic surfaces
cut out from $Y$ by arbitrary 3-dimensional projective subspaces in $\IP^5$.
With these preparations we can finally prove all parts of the main theorems.\\

{\noindent\bf Acknowledgements.} This project got launched when L.~Manivel pointed out
to one of us that the natural morphism $M_3(Y)\to \Grass(6,4)$ admits
a Stein factorisation $M_3(Y)\to Z_{\text{Stein}}\to \Grass(6,4)$ such that
$Z_{\text{Stein}}\to \Grass(6,4)$ has degree 72. We are very grateful to him
for sharing this idea with us. We have profited from discussions
with  C.~von Bothmer, I.~Dolgachev, E.~Looijenga and L.~Manivel.
The first named author was supported by the ANR program VHSMOD, Grenoble, and the 
Labex Irmia, Strasbourg.
The third author 
would like to thank the SFB Transregio 45 Bonn-Mainz-Essen and the 
Max-Planck-Institut für Mathematik Bonn for their hospitality.

\tableofcontents

\Section{Hilbert schemes of generalised twisted cubics}
\label{sec:HilbertSchemes}

A rational normal curve of degree 3, or twisted cubic for short, is a smooth
curve $C\subset\IP^3$ that is projectively equivalent to the image of $\IP^1$
under the Veronese embedding $\IP^1\to \IP^3$ of degree $3$. The set of all
twisted cubics is a 12-dimensional orbit under the action of $\LiePGl_4$. Piene
and Schlessinger \cite{PieneSchlessinger} showed that its closure $H_0$ is a 
smooth 12-dimensional component of $\Hilb^{3n+1}(\IP^3)$ and that the full
Hilbert scheme is in fact scheme theoretically the union of $H_0$ and a 
15-dimensional smooth variety $H_1$ that intersect transversely along a 
smooth divisor $J_0\subset H_0$. The second component $H_1$ parameterises plane
cubic curves together with an additional and possibly embedded point; it
will play no further r\^ole in our discussion. 

We will refer to any subscheme $C\subset\IP^3$ that belongs to a point in $H_0$ 
as a {\sl generalised twisted cubic} and to $H_0$ as the Hilbert scheme of 
generalised twisted cubics on $\IP^3$. 

There is an essential difference between curves parameterised by $H_0\setminus J_0$
and those parameterised by $J_0$. This difference is crucial for almost all arguments
in this article and enters all aspects of the construction. We therefore recall
the following facts from the articles of Ellingsrud, Piene, Schlessinger and
Str\o mme \cite{PieneSchlessinger, EllingsrudPieneStromme,EllingsrudStromme} in
some detail.
\newcounter{ourenumi}
\begin{list}{(\arabic{ourenumi})}{\usecounter{ourenumi}\setlength{\leftmargin}{7mm}}
\item  \label{item:aCMSequence} 
Curves $C$ with $[C]\in H_0\setminus J_0$ are arithmetically Cohen-Macaulay 
(aCM), i.e.\ their affine cone in $\IC^4$ is Cohen-Macaulay at the origin. The 
homogeneous ideal of such a curve is generated by a net of quadrics $(q_0,q_1,q_2)$ 
that arise as minors of a $3\times 2$-matrix $A_0$ with linear entries. There is 
an exact sequence 
\begin{equation}\label{eq:aCMSequence}
0\to \ko_{\IP^3}(-3)^{\oplus 2}\xra{\;A_0\;}
\ko_{\IP^3}(-2)^{\oplus 3}\xra{\Lambda^2A_0^t} \ko_{\IP^3}\lra \ko_C\lra 0.
\end{equation}
Up to projective equivalence there are exactly 8 isomorphism types of aCM-curves
represented by the following matrices:
$$
\begin{array}{cccc}
A^{(1)}=\left(\begin{smallmatrix}x_0&x_1\\x_1&x_2\\x_2&x_3\end{smallmatrix}\right),&
A^{(2)}=\left(\begin{smallmatrix}x_0&0\\x_1&x_2\\x_2&x_3\end{smallmatrix}\right),&
A^{(3)}=\left(\begin{smallmatrix}x_0&0\\x_1&x_2\\0&x_3\end{smallmatrix}\right),&
A^{(4)}=\left(\begin{smallmatrix}x_0&0\\x_1&x_1\\0&x_3\end{smallmatrix}\right)\phantom{.}\\[2ex]
A^{(5)}=\left(\begin{smallmatrix}x_0&0\\x_1&x_0\\x_2&x_3\end{smallmatrix}\right),&
A^{(6)}=\left(\begin{smallmatrix}x_0&0\\x_1&x_0\\0&x_3\end{smallmatrix}\right),&
A^{(7)}=\left(\begin{smallmatrix}x_0&0\\x_1&x_0\\x_2&x_1\end{smallmatrix}\right),&
A^{(8)}=\left(\begin{smallmatrix}x_0&0\\x_1&x_0\\0&x_1\end{smallmatrix}\right).
\end{array}
$$
The dimension of the corresponding strata in $H_0$ are $12$, $11$, $10$, $9$, $9$,
$8$, $7$ and $4$ in the given order. 
$A^{(1)}$ defines a smooth twisted cubic, $A^{(2)}$ the union of a smooth plane 
conic and a line, and $A^{(3)}$  a chain of three lines. These three types are
local complete intersections. $A^{(4)}$ defines the union of three collinear but
not coplanar lines. The matrices in the second row define non-reduced curves that
contain a line with multiplicity $\geq 2$, but are always purely 1-dimensional.
\item Curves $C$ with $[C]\in J_0$ are not Cohen-Macaulay (non-CM).
The homogeneous ideal of such a curve $C$ is generated by three quadrics, which 
in appropriate coordinates can be written as $x_0^2,x_0x_1,x_0x_2$, and a cubic
polynomial 
$h(x_1,x_2,x_3)=x_1^2 a(x_1,x_2,x_3)+x_1x_2b(x_1,x_2,x_3)+x_2^2c(x_1,x_2,x_3)$. 
The latter defines a cubic curve in the plane $\{x_0=0\}$ with a singularity at
the point $[0:0:0:1]$. Note that the three quadratic generators still arise
as minors of a $3\times 2$-matrix, namely $A_0=\left(\begin{smallmatrix}
0&-x_0&x_1\\x_0&0&-x_2\end{smallmatrix}\right)^t$. There is an exact sequence
{\small
$$
0\to \ko_{\IP^3}(-4)\to\ko_{\IP^3}(-3)^3\oplus \ko_{\IP^3}(-4)
\to \ko_{\IP^3}(-2)^3\oplus \ko_{\IP^3}(-3)\to \ko_{\IP^3}\to\ko_C\to 0.
$$
}Up to projective equivalence there are 9 isomorphism types of non-CM curves:
The generic $11$-dimensional orbit is represented by a nodal curve with polynomial 
$h=x_1^3+x_2^3+x_1x_2x_3$, and the $6$-dimensional unique closed orbit by a line 
with a planar triple structure defined by $h=x_1^3$. 
\end{list}
In each case, the linear span of $C$ is the ambient space $\IP^3$. Because of 
this it is easy to see that for any $m\geq 3$ the Hilbert scheme 
$\Hilb^{3n+1}(\IP^m)$ contains a smooth component $\Hilb^{gtc}(\IP^m)$ that 
parameterises gene\-ralised twisted cubics and that fibres locally trivially
over the Grassmannian variety of $3$-spaces in $\IP^m$. The morphism
$$s:\Hilb^{gtc}(\IP^m)\to \Grass(\IC^{m+1},4)$$ 
maps a generalised twisted cubic in $\IP^m$ to the projective $3$-space 
$\langle C\rangle$ spanned by $C$. Conversely, if $[p]\in\Grass(\IC^{m+1},4)$
is a point represented by an epimorphism $p:\IC^{m+1}\to W$ onto a four-dimensional
vector space $W$, or equivalently, by a three\-dimensional space $\IP(W)\subset\IP^m$,
then the fibre $s^{-1}([p])$ is the Hilbert scheme of generalised 
twisted cubics in $\IP(W)$. Clearly, $\dim \Hilb^{gtc}(\IP^m)=4m$. For any projective 
scheme $X\subset\IP^m$ let $\Hilb^{gtc}(X):=\Hilb^{3n+1}(X)\cap \Hilb^{gtc}(\IP^m)$ 
denote the {\sl Hilbert scheme of generalised twisted cubics on $X$}. 

Let $\kc\subset \Hilb^{gtc}(\IP^5)\times \IP^5$ denote the universal family of 
generalised twisted cubics and let $\pr_1$ and $\pr_2$ be the projections to 
$\Hilb^{gtc}(\IP^5)$ and $\IP^5$, respectively. It follows from 
\cite{EllingsrudStromme},\setlength{\leftmargin}{0mm} Cor.\ 2.4., that the sheaf 
$\ka:=\pr_{1*}(\ko_{\kc}\tensor\pr_2^*\ko_{\IP^5}(3))$ is locally free of rank 10 \BB
and that the natural restriction homomorphism 
$\varepsilon: S^3\IC^6\tensor \ko_{\Hilb^{gtc}(\IP^5)}\to \ka$ is surjective. 
Let $f\in S^3\IC^6$ be a non-zero homogeneous polynomial of degree 3 and 
$Y=\{f=0\}$ the corresponding cubic hypersurfaces. Then the Hilbert scheme
\begin{equation}
M_3(Y):=\Hilb^{gtc}(Y)
\end{equation}
of generalised twisted cubic curves on $Y$ is scheme theoretically isomorphic 
to the vanishing locus of the section $\varepsilon(f)\in H^0(\Hilb^{gtc}
(\IP^5),\ka)$. In particular, any irreducible component of $M_3(Y)$ is at least 
$10$-dimensional.

A simple dimension count shows that the set of cubic polynomials in six variables
that vanish along a plane is 55 dimensional and hence a divisor in \BB
the 56-dimensional space of all cubic polynomials. We will from now on impose the condition 
that $Y$ is smooth and does not contain a plane. As we will show in 
Section \ref{sec:SmoothnessAndIrreducibility} this implies that $M_3(Y)$ is smooth 
as well.

To simplify the notation we put $\IG:=\Grass(\IC^6,4)$. Closed points in $\IG$
parameterise epimorphisms $p:\IC^6\to W$ or equivalently 3-dimensional linear 
subspaces $\IP(W)\subset \IP^5$. Since a smooth cubic hypersurface cannot contain 
a $3$-space, the intersection $S=\IP(W)\cap Y$ is a cubic surface in $\IP(W)$, 
and since $Y$ does not even contain a plane, the surface $S$ is reduced and 
irreducible, i.e.\ integral.

By construction, $M_3(Y)=\Hilb^{gtc}(Y)$ comes equipped with a morphism 
$$s:\Hilb^{gtc}(Y) \to \IG, \quad 
[C\subset Y]\mapsto [\langle C\rangle\subset\IP^5],$$
with fibres
$$s^{-1}([p])=\Hilb^{gtc}(S),\quad\quad S=Y\cap \IP(W).$$

\Section{Twisted cubics on cubic surfaces}
\label{sec:TwistedCubicsOnCubicSurfaces}

Since the morphism $s:\Hilb^{gtc}(Y)\to \IG$ constructed at the end of the
previous paragraph has fibres of the form $\Hilb^{gtc}(S)$, where $S$ is
an integral cubic surface, we will study these Hilbert schemes for arbitrary
integral cubic surfaces abstractly and quite independently of $Y$.  

Cubic surfaces form a classical subject of algebraic geometry. The classification
of the different types of singularities was given by Schläfli \cite{Schlaefli} in 
1864. A classical source of information on cubic surfaces is the book of 
Henderson \cite{Henderson}. For treatments in modern terminology see 
the papers of Looijenga \cite{Looijenga} and Bruce and Wall \cite{BruceWall}. We 
refer to the book of Dolgachev \cite{Dolgachev}, Ch. 9, and the seminar notes of 
Demazure \cite{Demazure} for further references and all facts not proved here. 
A cubic surface $S\subset\IP^3$ belongs to one of the following four classes: 
\begin{enumerate}
\item $S$ has at most rational double point singularities, 
\item $S$ has a simple-elliptic singularity, 
\item $S$ is  integral but not normal, or 
\item $S$ is not integral, i.e.\ its defining polynomial is reducible.
\end{enumerate}
Let $\IB:=\IP(S^3\IC^4{}^*)$ denote the $19$-dimensional moduli space of {\sl 
embedded} cubic surfaces, and let $\IB^{\ganz}\subset\IB$ denote the open subset 
of integral surfaces. 
It is stratified by locally closed subsets $\IB(\Sigma)$, where $\Sigma$ is a 
string describing the common singularity type of the surfaces $[S]\in \IB(\Sigma)$. 
For example, 
$\IB(A_1+2A_2)$ will denote the $5$-codimensional stratum of surfaces with one 
$A_1$ and two $A_2$-singularities, whereas the $7$-codimensional stratum 
$\IB(\tilde E_6)$ parameterises surfaces 
with a simple-elliptic singularity. For most singularity types, the stratum 
$\IB(\Sigma)$ is a single $\LiePGl_4$-orbit with the exception of $\Sigma=\emptyset,
A_1$, $2A_1$, $3A_1$, $A_2$, $A_1+A_2$ and $\tilde E_6$. In these cases, the 
isomorphism type is not determined by the singularity type. The moduli problem 
for {\sl isomorphism types} of cubic surfaces is treated by Beauville in 
\cite{Beauville} in terms of geometric invariant theory.

\subsection{Cubic surfaces with rational double points}

Let $S\subset\IP^3$ be a cubic surface with at most rational double point 
singularities and let $\sigma:\tilde S\to S$ be its minimal resolution. The 
canonical divisors of $S$ and $\tilde S$ are $K=-H$, if $H$ denotes a 
hyperplane section, and $\tilde K=-\sigma^*H$, since $\sigma$ is crepant. In 
fact, $\sigma$ is defined by the complete anti-canonical linear system 
$|-\tilde K|$. The smooth surface $\tilde S$ is an almost (or weak) Del Pezzo 
surface. The orthogonal complement $\Lambda:=\tilde K^{\perp}\subset 
H^2(\tilde S;\IZ)$ of the canonical divisor is a negative definite root lattice 
of type $E_6$. The components $E_1,\ldots,E_m$ of the exceptional divisor of 
$\sigma$ are $-2$-curves whose classes $\alpha_1,\ldots,\alpha_m$ form a subset 
$\Delta_0$ in the root system $R\subset\Lambda$ that is a root basis for a 
subsystem $R_0\subset R$. Let $\Lambda_0\subset\Lambda$ denote the corresponding
sub-lattice. Configurations $\Lambda_0\subset\Lambda$ are classified by 
subdiagrams of the extended Dynkin diagram $\tilde E_6$ (cf.\ \cite{Bourbaki} 
exc.~4.4, p.~126, or \cite{Wall}, Thm.\ 2B.). That all lattice theoretically admissible
configurations also arise geometrically was shown in \cite{Looijenga}. (As
Looijenga pointed out to us, the equivalent statement is not true for 
the other simple elliptic singularities.)
The connected components of the Dynkin diagram of $R_0$ are in bijection with the 
singularities of $S$. This limits the possible combinations of singularity 
types of $S$ to the following list: $A_1$, $2A_1$, $A_2$, $3A_1$, $A_1+A_2$, 
$A_3$, $4A_1$, $2A_1+A_2$, $A_1+A_3$, $2A_2$, $A_4$, $D_4$, $2A_1+A_3$, 
$A_1+2A_2$, $A_5$, $D_5$, $A_1+A_5$, $3A_2$, $E_6$.

It is classically known that there is a close connection between roots in the
$E_6$-lattice of the resolution $\tilde S$, twisted cubics on $S$ and representations 
of the cubic equation of $S$ as a linear determinant, and we will further exploit this 
connection in Section \ref{sec:ModuliOfDeterminantalRepresentations}. 
We refer to the book of Dolgachev \cite{Dolgachev} for further information. 
We could, however, not find a reference for the r\^ole of the Weyl group in this 
context and therefore include a detailed discussion here. We also 
take the occasion (cf.\ Table 1 in Sec.\ \ref{subsec:determinantaldescriptions}) 
to correct Table 9.2. in 
\cite{Dolgachev}, where this action was overlooked.

Let $W(R_0)$ denote the subgroup of the Weyl group $W(R)$ that is generated by 
the reflections $s_i$ in the effective roots $\alpha_i$, $i=1,\ldots,m$. The 
root system $R$ decomposes into finitely many orbits with respect to this action. 
The orbits contained in $R_0$ are exactly the irreducible components of $R_0$ 
and are therefore in bijection with the singularities of $S$. It is a well-known 
property of root systems that every $W(R_0)$-orbit of $\Lambda_0\tensor\IQ$ meets 
the closed Weyl chamber $\overline\kc=\{\beta\;|\; \beta.\alpha_i\leq 0\}$ (and the
opposite chamber $-\bar \kc$) exactly 
once (cf. \cite{Humphreys} Thm.\ 1.12). If we apply this to the orthogonal 
projection of any root $\alpha$ to $\Lambda_0\tensor\IQ$ we find in every 
$W(R_0)$-orbit $B\subset R$ unique roots $\alpha^+_B$ and $\alpha_B^-$ that are 
characterised by the property $\pm\alpha^\pm_B.\alpha_i\leq 0$ for $i=1,\ldots,m$. 
We will refer to $\alpha^+_B$ and $\alpha^-_B$ as the {\sl maximal} resp.\ 
{\sl minimal} root of the orbit. Note that $-\alpha_B^+$ equals $\alpha_B^-$ only
if $B=-B$, i.e.\ if $B$ is a subset of $R_0$. If $R_p$ is the irreducible subsystem 
of $R_0$ that corresponds to a singularity $p\in S$, then $\alpha^+_{R_p}$ is the 
longest root in the root system $R_p$ with respect to the root basis given by 
exceptional curves in the fibre of $p$. It also equals the cohomology class of
the fundamental cycle $Z_p$ as defined by Artin \cite{Artin}.

\begin{theorem}\label{thm:ADESurfaceCase} --- 
Let $S$ be a cubic surface with at most rational double
point singularities. Then
$$\Hilb^{gtc}(S)_{\red}\isom \coprod_{B\in R/W\!(R_0)}|\ko_{\tilde S}(\alpha_B^--\tilde K)|
\;\isom\;({R/W\!(R_0)})\times \IP^2.$$
Moreover, an orbit $B$ corresponds to families of non-CM or aCM-curves depending
on whether $B$ contains effective roots or not. The generic curve in a linear
system of aCM curves is smooth.
\end{theorem}
 
Some components of $\Hilb^{gtc}(S)$ can be non-reduced, as can be easily seen from
the fact that the morphism $\Hilb^{gtc}(Y)\to \IG$ is ramified along the
divisor in $\IG$ that corresponds to singular surfaces. For the purpose of this
article there is no need to discuss this question in any detail.

We will prove the theorem in several steps.

\begin{proposition}\label{prop:MaximalRootsGiveGTC}--- 
1. Let $C\subset S$ be a generalised twisted cubic, and let $\tilde C=\sigma^{-1}(C)
\subset \tilde S$ denote the scheme theoretic inverse image. Then $\tilde C$ is 
an effective divisor such that the class of $\tilde C+\tilde K$ is a root in $R$.
This root is the maximal root in its orbit. Moreover, $\sigma_*\ko_{\tilde C}=\ko_C$.\\
2. Conversely, let $\alpha$ be a maximal root and let $\tilde C\in 
|\alpha-\tilde K|$. Then $C:=\sigma(\tilde C)\subset S$ is a subscheme with
Hilbert polynomial $3n+1$. 
\end{proposition}
 
\begin{proof} Ad 1: Let $I\subset \ko_S$ and $\tilde I\subset \ko_{\tilde S}$ denote
the ideal sheaves of $C$ and $\tilde C$, respectively, so that $\sigma^*I\epimorph 
\tilde I$ and $I\subset\sigma_*\tilde I$. For any singular point $p\in S$, there
is an open neighbourhood $U$ and an epimorphism $\ko_U^n\epimorph I|_U$. This 
induces surjective maps $\ko_V^n\to \sigma^*I|_V\to \tilde I|_V$ on a neighbourhood 
$V=\sigma^{-1}(U)$ of the fibre $\sigma^{-1}(p)$. As $\sigma$ has at most 
$1$-dimensional fibres, all second or higher direct images of coherent sheaves 
on $\tilde S$ vanish, and pushing down the epimorphism $\ko_V^n\to \tilde I|_V$ 
along $\sigma$ yields an epimorphism $(R^1\sigma_*\ko_{\tilde S})^n|_U\to 
R^1\sigma_*\tilde I|_U$. Since $S$ has rational singularities, 
$R^1\sigma_*\ko_{\tilde S}=0$ and so $R^1\sigma_*\tilde I=0$. This implies that 
in the following commutative diagram both rows are exact, that $\alpha$ is 
injective and that $\beta$ is surjective:
\begin{equation}
\begin{array}{ccccccccc}
0&\lra&\sigma_*\tilde I&\lra&\sigma_*\ko_{\tilde S}&
\lra&\sigma_*\ko_{\tilde C}&\lra&0\\
&&{\scriptstyle \alpha}\Big\uparrow{\scriptstyle\phantom\alpha}
&&\Big\|
&&{\scriptstyle \beta}\Big\uparrow{\scriptstyle\phantom\beta}\\
0&\lra&I&\lra&\ko_{S}&\lra&\ko_{C}&\lra&0\\
\end{array} 
\end{equation}
The homomorphism $\beta$ is generically an isomorphism. If $C$ has no embedded 
points, $\beta$ is an isomorphism everywhere. In this case $\tilde C$ cannot have embedded 
points either as they would show up as embedded points in 
$\sigma_*\ko_{\tilde C}$. Hence $\tilde C$ is an effective divisor. 

If on the other hand $C$ has an embedded point at $p$ then $C$ is a non-CM curve, 
and it follows from the global structure of such curves that $p$ is a singular 
point of $S$, say with ideal sheaf $\gothm$, and that $I$ is of the form 
$\gothm\cdot \ko_{S}(-H)$ for a hyperplane section $H$ through $p$. Let $Z_p$ 
denote the fundamental cycle supported on the exceptional fibre $\sigma^{-1}(p)$. 
By Artin's Theorem 4 in \cite{Artin}, $\sigma^*\gothm\cdot\ko_{\tilde S}=\ko_{\tilde S}(-Z_p)$ 
and $\sigma_*\ko_{\tilde S}(- Z_p)=\gothm$, so that 
$\tilde I=\ko_{\tilde S}(-Z_p-\sigma^*H)$ and $I=\sigma_*\tilde I$. 

Thus $\tilde C$ is always an effective divisor and $\sigma_*\ko_{\tilde C}=\ko_C$. 
Since $R^i\sigma_*\ko_{\tilde S}=0$ and $R^i\sigma_*\tilde I=0$ for $i>0$ one 
also gets $R^i\sigma_*\ko_{\tilde C}=0$ for $i>0$, and $\chi(\ko_{\tilde C})=
\chi(\ko_C)=1$. 

Since $\tilde C.(-\tilde K)=C.H=3$, an application of the Riemann-Roch-formula 
gives $(\tilde C)^2=1$ and hence $(\tilde C+\tilde K).\tilde K
=0$ and $(\tilde C+\tilde K)^2=-2$. This shows that $\alpha:=\tilde C+\tilde K$ is a 
root in the lattice $\Lambda$. Since the ideal sheaf $\tilde I=\ko_{\tilde S}(-\tilde C)
=\ko_{\tilde S}(-\alpha+\tilde K)$ is generated by global sections in a neighbourhood
of every effective $(-2)$-curve $E$ one gets $\alpha.E=-\deg(\tilde I|_E)\leq 0$. 
This shows that $\alpha$ is the maximal root of its orbit.

Ad 2: Taking direct images of $0\to \ko_{\tilde S}(-\tilde C)\to \ko_{\tilde S}
\to \ko_{\tilde C}\to 0$ one gets an exact sequence
$0\to I_C\to \ko_S\to \pi_*\ko_{\tilde C}\to R^1\sigma_*\ko_{\tilde S}(-\tilde C)\to 0$,
where $I_C$ is the ideal sheaf of $C$, and all other higher direct image sheaves vanish.
As $\alpha$ is maximal, the restriction of $\ko_{\tilde S}(-\tilde C)$ to any exceptional
curve has non-negative degree. Let $Z$ denote the sum of the fundamental cycles of all
exceptional fibres. According to \cite{Artin}, Lemma 5, one has 
$H^1(Z, \ko_{\tilde S}(-\tilde C-mZ))=0$ for all $m\geq 0$, and the Theorem on Formal
Functions \cite{EGA}, Prop.\ III.4.2.1, now yields $R^1\sigma_*(\ko_{\tilde S}(-\tilde C))=0$ 
and thus $\sigma_*\ko_{\tilde C}=\ko_C$. It follows that 
\begin{eqnarray*}
\chi(\ko_C(nH))&=&\chi(\ko_{\tilde C}(-n\tilde K))\;=\;
\chi(\ko_{\tilde S}(-n\tilde K))-\chi(\ko_{\tilde S}
(-\tilde C-n\tilde K))\\
&=&\sfrac12\Big( n(n+1)\tilde K^2-(-\tilde C-n\tilde K)(-\tilde C-(n+1)\tilde K)\Big)\\
&=&\sfrac12\Big(-\tilde C^2+(2n+1)\tilde C(-\tilde K)\Big)\;=\;3n+1.
\end{eqnarray*}
\end{proof}

The intersection product of an irreducible curve $D\subset\tilde S$ with 
$-\tilde K$ can only take the following values: Either $(-\tilde K)\cdot D=0$, 
in which case $D$ is an exceptional $(-2)$-curve, or $(-\tilde K)\cdot D=1$, 
which implies that the image of $D$ in $S$ is a line, so that $D$ itself must 
be a smooth rational curve with $D^2=-1$, or, finally, $(-\tilde K)\cdot D\geq 2$ and 
$D^2\geq 0$. 

\begin{lemma}\label{lem:keyestimate}--- If $\alpha$ is a minimal root, 
then $(\alpha-\tilde K)\cdot F\geq 0$ for every effective divisor $F$ with 
$F\cdot(-\tilde K)\leq 1$.
\end{lemma}

\begin{proof} $F$ is the sum of $(-2)$-curves and at most one $(-1)$-curve. As 
$\alpha$ is minimal it intersects each $(-2)$-curve non-negatively. It suffices 
to treat the case that $F$ is a $(-1)$-curve. But then $u=\sfrac13\tilde K+F$
lies in $\Lambda\tensor \IQ$ with $u^2=-\sfrac43$. Now
$(\alpha-\tilde K).F=\alpha\cdot u+1$, so by Cauchy-Schwarz we get
$(\alpha-\tilde K).F\geq 1-\sqrt{2}\sqrt{\frac4 3}>-\sfrac23$. But the left
hand side is an integer. 
\end{proof}

\begin{lemma}\label{lem:MinimalRoot}--- Let $\alpha$ be a minimal root. Then 
the linear system $|\alpha-\tilde K|$ is two-dimensional and base point free. 
In particular, the generic element in  $|\alpha-\tilde K|$ is a smooth rational 
curve.
\end{lemma}

\begin{proof} Let $L_\alpha=\ko_{\tilde S}(\alpha-\tilde K)$. Since 
$(2\tilde K-\alpha)\cdot(-\tilde K)=-6<0$, the divisor $2\tilde K-\alpha$ cannot
be effective. This shows that $h^2(L_\alpha)=h^0(\ko(2\tilde K-\alpha))=0$.
Any irreducible curve $D$ with $0>\deg\,L(\alpha)|_D=(\alpha-\tilde K)D$ must 
be a fixed component of the linear system $|\alpha-\tilde K|$ satisfying 
$D^2<0$ and hence $D(-\tilde K)\leq 1$. But this contradicts Lemma 
\ref{lem:keyestimate}. Hence $L_\alpha$ is nef and even big, and a fortiori 
$L_\alpha(-\tilde K)$ is as well. The Kawamata-Viehweg Vanishing Theorem
now implies that $h^1(L_\alpha)=0$, and Riemann-Roch gives $h^0(L_\alpha)=3$.

Suppose that $F$ is the fixed component of $|\alpha-\tilde K|$ and $M$ a 
residual \BB irreducible curve. Then $M$ is effective and nef, and $M-\tilde K$
is big and nef. This implies that $h^i(\ko(M))=0$ for $i>0$ and $\chi(\ko(M))=h^0(\ko(M))
=h^0(L_\alpha)=3$. Now Riemann-Roch gives $M^2=4-M(-\tilde K)=1+F(-\tilde K)\geq 1$. 
As $M$ cannot be a $(-1)$ or $(-2)$ curve, we have $M(-\tilde K)\geq 2$ and 
$F(-\tilde K)\leq 1$. By Lemma \ref{lem:keyestimate} we get
$1=(\alpha-\tilde K)^2=(\alpha-\tilde K)F+FM+M^2\geq M^2$. This shows in turn
$M^2=1$, $FM=0$, $F^2=0$ and $F(-\tilde K)=0$. Since $\Lambda$ is negative 
definite, $F=0$. This shows that $|\alpha-\tilde K|$ has no fixed component.

Since $(\alpha-\tilde K)^2=1$, there is at most one base point $p$. If there were
such a point, consider the blow-up $\widehat S\to \tilde S$ at $p$ with exceptional
divisor $E$. The linear system $-\widehat K=-\tilde K-E$ is effective, big and 
nef, and since $|\alpha-\tilde K-E|$ has not fixed components either, another 
application of the Kawamata-Viehweg Vanishing Theorem gives the contradiction
$\IC=H^0(E, \ko(\alpha-\tilde K)|_E)\hookrightarrow 
H^1(\widehat S,\ko(\alpha-\tilde K-E))=0$.

The smoothness of a generic curve in the linear system follows from Bertini's theorem.
\end{proof}

\begin{proposition}\label{prop:aCMcase} --- Let $\alpha\in R\setminus R_0$, and let $\alpha^+$ 
and $\alpha^-$ denote the maximal and the minimal root, resp., of its orbit.
\begin{enumerate}
\item The linear system $|\alpha-\tilde K|$ is independent of the choice of $\alpha$
in its $W(R_0)$-orbit. More precisely, the differences $e_+=\alpha^+-\alpha$ and 
$e_-=\alpha-\alpha^-$ are sums of $(-2)$-curves, and the multiplication by these 
effective classes gives isomorphisms
$$|\alpha^--\tilde K|\xra{\;e_-\;}|\alpha-\tilde K|\xra{\;e_+\;}|\alpha^+-\tilde K|.$$
In particular, $\dim |\alpha-\tilde K|=2$. The linear system $|\alpha^--\tilde K|$ is 
base point free.
\item For every curve $\tilde C\in |\alpha^--\tilde K|$ one has $C:=\sigma(\tilde C)=
\sigma(\tilde C+e_-)$, and $C$ is a generalised twisted cubic.
\item The image $C=\sigma(\tilde C)$ of a generic curve $\tilde C\in |\alpha-\tilde K|$
is smooth.
\end{enumerate}
\end{proposition}

\begin{proof} As before, let $L_\alpha=\ko_{\tilde S}(\alpha-\tilde K)$.

Assume first that $\alpha^-\neq \alpha^+$, and let $\beta$ be any 
root from the orbit of $\alpha$, different from $\alpha^-$. Then there is an 
effective root $\alpha_i$ such that $\beta.\alpha_i\leq -1$. In fact, 
$\beta.\alpha_i=-1$, since $\beta.\alpha_i=-2$ implies $\beta=\alpha_i$ 
contradicting the assumption that no root of the orbit of $\alpha$ is effective.
Let $\beta'=\beta-\alpha_i=s_i(\beta)$ be the root obtained by reflecting 
$\beta$ in $\alpha_i$. Now multiplication with the equation of the exceptional 
$(-2)$-curve $E_i$ gives an exact sequence 
$0\to L_{\beta'}\to L_\beta\to L_\beta|_{E_i}\to 0$. Since $L_\beta|_{E_i}
=\ko_{E_i}(-1)$ has no cohomology, one gets $h^i(L_{\beta'})=
h^i(L_\beta)$ for all $i$. In particular, $|L_{\beta'}|\to |L_\beta|$ is an
isomorphism. If $\tilde C\in |L_{\beta'}|$, there is an exact sequence
$$0\to \ko_{\tilde S}(-\tilde C-E_i)\to\ko_{\tilde S}(-C)\to \ko_{E_i}(-1)\to 0,$$
so that the ideal sheaves 
$\sigma_*(\ko_{\tilde S}(-\tilde C-E_i))=\sigma_*(\ko_{\tilde S}(-C))\subset\ko_S$
define the same image curve $\sigma(\tilde C+E_i)=\sigma(\tilde C)$.
Replacing $\beta$ by $\beta'$ subtracts a fixed component from the linear system 
$|L_\beta|$. Iterations of this step lead in finitely many steps to the minimal 
root $\alpha_-$. The argument can be reversed to move in the opposite direction 
from $\beta$ to $\alpha^+$. 

Hence all roots in the $W(R_0)$-orbit of $\alpha$ define isomorphic linear systems 
and the same family of subschemes in $S$. Of course, if $\alpha^-=\alpha^+$, this
is true as well.

Taking $\alpha=\alpha^+$, it follows from Proposition 
\ref{prop:MaximalRootsGiveGTC} that these subschemes are generalised twisted 
cubics. Taking $\alpha=\alpha^-$, it follows from Lemma \ref{lem:MinimalRoot}
that the linear system is two-dimensional and that the generic curve $\tilde C
\in |L_{\alpha^-}|$ is smooth. If $p\in S$ is any singular point and $R_p\subset 
R_0\subset R$ the corresponding root system, the pre-image $\sigma^{-1}(p)$ equals 
the effective divisor corresponding to the maximal root $\alpha^+_{R_p}$. As
$\alpha^-.\alpha^+_{R_p}$ can only take the values $0$ or $1$, the curve $C
:=\sigma(\tilde C)$ has multiplicity $0$ or $1$ at $p$. Hence $p$ is a smooth
point of $C$ or no point of $C$ at all. As $\sigma$ is birational off the singular
locus of $S$, the scheme $C$ is a smooth curve. 
\end{proof}

The situation for effective roots is slightly different:

\begin{proposition}\label{prop:nonaCMcase} --- 
Let $p\in S$ be a singular point, let $R_p\subset 
R_0\subset R$ denote the corresponding irreducible root system with maximal root
$\alpha^+$ and minimal root $\alpha^-=-\alpha^+$. Let $\alpha\in R_p$
be an effective root.
\begin{enumerate}
\item The difference $e:=\alpha^+-\alpha$ is effective. Multiplication with the 
effective classes $e$, $\alpha$, and $e$, resp., induces the following isomorphisms
$$\IP^2\isom|\alpha^--\tilde K|\xra{\isom}|-\alpha-\tilde K|\subsetneq \IP^3\isom |-\tilde K| 
\xra{\isom} |\alpha-\tilde K|\xra{\isom} |\alpha^+-\tilde K|.$$
\item For every curve $\tilde C\in |\alpha^--\tilde K|$, the image $C=\sigma(\tilde C+2Z_p)$
is a generalised twisted cubic in $S$ with an embedded point at $p$, and every
non aCM-curve $C\subset S$ with an embedded point at $p$ arises in this way.
\end{enumerate}
\end{proposition}

\begin{proof} As long as $\beta\in R_p$ is a non-effective root the first part
of the proof of the previous proposition still holds and shows that $\beta-\alpha^-$
is effective, represented, say, by a curve $E'$, that multiplication with $E'$ 
defines an isomorphism $|\alpha^--\tilde K|\to |\beta-\tilde K|$ and that
for every curve $\tilde C\in |\alpha^--\tilde K|$ the divisors $\tilde C$
and $\tilde C+E$ have the same scheme theoretic image in $S$. 
The same method shows that for every effective root $\beta\in R_p$  the
linear systems $|\beta-\tilde K|$ and $|\alpha^--\tilde K|$ are isomorphic and
give the same family of subschemes in $S$.

Multiplication by the fundamental cycle $Z_p$ (of class $\alpha^+$) defines an 
embedding of the two-dimensional linear system $|-\alpha^--\tilde K|$ into the 
three-dimensional linear system $|-\tilde K|$ of hyperplane sections with respect 
to the contraction $\sigma:\tilde S\to S\subset\IP^3$. The image of the embedding 
is the linear subsystem of hyperplane sections through $p$. Let $\tilde C$ be
any curve in the linear system $|\alpha^--\tilde K|$. Its image $C_0=\sigma(\tilde C)$
is a hyperplane section $C_0=H\cap S$ for a hyperplane $H$ through $p$. Then
$\tilde C$ and $\tilde C+Z_p$ have the same image $C$, but $\sigma(C+2Z_p)$ 
has an additional embedded point at $p$. By Proposition \ref{prop:MaximalRootsGiveGTC},
the image is a generalised twisted cubic, necessarily of non-CM type.
\end{proof}

The Propositions \ref{prop:aCMcase} and \ref{prop:nonaCMcase} together
imply Theorem \ref{thm:ADESurfaceCase}.

\subsection{Cubic surfaces with a simple-elliptic singularity}

Simple-elliptic singularities were introduced and studied in general by Saito in 
\cite{Saito} and further studied by Looijenga \cite{Looijenga}. 
A cubic surface with a simple-elliptic singularity is a cone over 
a smooth plane cubic curve $E\subset \IP^2\subset\IP^3$ with a vertex $p\in 
\IP^3\setminus \IP^2$. The type of such a simple-elliptic singularity
is denoted by $\tilde E_6$. 

In appropriate coordinates $x_0,\ldots,x_3$ the surface $S$ is given by the
vanishing of $g=x_1^3+x_2^3+x_3^3-3\lambda x_1x_2x_3$ for some parameter 
$\lambda\in \IC$, $\lambda^3\neq 1$. The same equation defines
a smooth elliptic curve $E$ in the plane $\{x_0=0\}$, and $S$ is the cone over
$E$ with vertex $p=[1:0:0:0]$. The parameter $\lambda$ determines the $j$-invariant
of the curve $E$. The Jacobian ideal of $g$ in the local ring $\ko_{S,p}$ is 
generated by the quadrics $x_1^2-\lambda x_2x_3$, $x_2^2-\lambda x_1x_3$, 
$x_3^2-\lambda x_1x_2$. The monomials $1,x_1,x_2,x_3,x_1x_2,x_1x_3,x_2x_3,
x_1x_2x_3$ form a basis of $\ko_{S,p}/J(g)$ and hence of the tangent space to 
the deformation space of the singularity. Since the total degree of all monomials 
is $\leq 3$, all deformations are realised by deformations of $g$ in the space 
of cubic polynomials. This shows that $\IB$ is the 
base of a versal deformation for the singularity of $S$. Note that although 
the Milnor ring $\ko_{S,p}/J(g)$ is 8-dimensional the stratum $\IB(\tilde E_6)$ 
has codimension 7 since the parameter corresponding to the monomial 
$x_1x_2x_3$ only changes the isomorphism type of the elliptic curve. 

\begin{proposition}\label{prop:simpleellipticfibres}--- Let $S\subset\IP^3$ be 
the cone over a plane elliptic curve $E$ with vertex $p$. Then 
$$\Hilb^{gtc}(S)_{\red}\isom \Sym_3(E)=E^3/S_3,$$
the third symmetric product of $E$. If $q=[q_1+q_2+q_3]\in \Sym_3(E)$ is not 
a collinear triple, the corresponding generalised twisted cubic is the union of 
the three lines connecting $p$ with each $q_i$. If $q=E\cap H$ for a hyperplane 
$H$ through $p$, the generalised twisted cubic is $H\cap S$ with an embedded point 
at $p$. The addition map $\Sym_3(E)\to E$ is a $\IP^2$-bundle, and the non-CM
curves in $\Hilb^{gtc}(S)$ form the fibre over the zero element $0\in E$.
\end{proposition}

\begin{proof} The only irreducible rational curves on $S$ are lines connecting 
the vertex $p$ with a point $q\in E$. Let $C$ be the union of three such lines 
over possibly coinciding points $q_1,q_2,q_3\in E$. The Hilbert polynomial of $C$ 
is $3n+1$ unless the points are collinear: the Hilbert polynomial then drops by 
one to $3n$. In this case, one has to augment $C$ by an embedded point at $p$.  
\end{proof}

\subsection{Non-normal integral cubic surfaces}

Assume that the cubic surface $S$ is irreducible and reduced, but not normal. 
Then $S$ is projectively equivalent to one of four surfaces given by the 
following explicit equations:  
$$\begin{array}{ll}
X_6=\{t_0^2t_2+t_1^2t_3=0\},&
X_7=\{t_0t_1t_2+t_0^2t_3+t_1^3=0\},\\
X_8=\{t_1^3+t_2^3+t_1t_2t_3=0\},&
X_9=\{t_1^3+t_2^2t_3=0\}. 
\end{array}$$
The labelling is chosen in such a way that in each
case the stratum $\IB(X_n)$ is a single $\LiePGl_4$-orbit of codimension $n$ in $\IB$.
Moreover, each $X_m$ lies in the closure the orbit of $X_{m-1}$. 

In fact, the mutual relation between these strata can be made explicit:
Both $\IB(X_9)$ and $\overline{\IB(X_6)}$ are smooth. A slice $F$ in 
$\overline{\IB(X_6)}$ to $\IB(X_9)$ through the point $X_9$ is three-dimensional.
One such slice, or more precisely, the family of non-normal surfaces parameterised
by it, is
$$\tilde f=t_1^3+t_2^2t_3+a t_1^2t_3+b t_0t_1t_2+ct_0t_1^2,\quad (a,b,c)\in\IC^3.$$
The discriminant of this family is $\Delta=ab^2+c^2$. One obtains the following
stratification: $\tilde f_{a,b,c}$ defines a surface isomorphic to 
$$
\left\{
\begin{array}{c}
X_9,\\X_8,\\X_7,\\X_6, 
\end{array}\right.
\text{ if }
\left\{
\begin{array}{l}
a=b=c=0,\\
a\neq 0,\; b=c=0,\\
\Delta=0, b\neq 0,\\
\Delta\neq 0.
\end{array}
\right.
$$
In particular, there are three different types of $X_6$ surfaces over the real numbers
corresponding to the components of the complement of the Whitney-umbrella $\{\Delta=0\}$.

We will now describe $\Hilb^{gtc}(X_8)$; the other cases can be treated similarly.
The surface $S=X_8$ is a cone in $\IP^3$ over a plane nodal cubic. Its normalisation 
$\tilde S$ is a cone in $\IP^4$ over a smooth twisted cubic $B$ in a hyperplane 
$U\subset\IP^3$. Let $v$ denote the vertex of $\tilde S$. The normalisation morphism 
$\nu:\tilde S\to S$ is the restriction to $\tilde S$ of a central projection 
$\IP^4\dashrightarrow \IP^3$ with centre in a point $c$ on a secant line $L$ of 
$B$. Finally, let $\hat S\to \tilde S$ denote the minimal resolution of the 
singularity of $\tilde S$. The exceptional curve $E$ is a rational curve with 
self intersection $-3$, and $\hat S$ is isomorphic to Hirzebruch surface $\IF_3$. 
Lines in $\tilde S$ through the vertex $v$ correspond to fibres $F$ of the ruling 
$\hat S\to \IP^1$, and both $E$ and $B$ are sections to this fibration. Any 
generalised twisted cubic on $S$ when considered as a cycle, arises as the image 
of a divisor on $\hat S$ of degree $3$ with respect to $E+3F$. Now, the only 
irreducible curves of degree $\leq 3$ on $\hat S$ belong to the linear systems 
$|E|$, $|F|$, $|E+3F|$ (cf.\ \cite{Hartshorne}). 
As $E$ is contracted to a point in $\tilde S$, it suffices to consider the curves 
in $|E+3F|=:P\isom \IP^4$. Note that $P$ is the dual projective space to the 
$\IP^4$ containing $\tilde S$. The images in $\tilde S$ of the the curves in 
the linear system $|E+3F|$ are exactly the hyperplane sections. Let 
$T\subset \IP^4$ denote the plane through the line $L$ and the vertex $v$, 
and let $T^\perp\subset P$ denote the dual line. The plane $T$ intersects $\tilde S$ 
in two lines $F_0$ and $F_\infty$ which are glued to a single line $F'$ in $S$ by 
the normalisation map. So far we have identified the underlying cycles of a 
generalised twisted cubics on $S$ as images of hyperplane sections of $\tilde S$: 
they are parameterised by $P$. In order to get the scheme structures as well, we 
need to blow-up $P$ along $T^\perp$. The fibres of the corresponding fibration 
$P':=\Bl_{T^\perp}(P)\to T^*$ have the following description: If $[M]\in T^*$ is 
represented by a line $M\subset T$, the fibre over $[M]$ is the $\IP^2$ of all 
hyperplanes in $\IP^4$ that contain $T$. It is clear that the families of
hyperplanes through the lines $F_0$ and $F_\infty$ parameterise the same curves in 
$S$. Identifying $[F_0]$ and $[F_\infty]$ in $T^*$ and the corresponding fibres
in $P'$ we obtain non-normal varieties $T^\dagger:=T^*/\sim$ and $P^\dagger/\sim$
with a natural $\IP^2$-fibration $P^\dagger\to T^\dagger$. It is not difficult
to explicitly describe the family of curves parameterised by $P^\dagger$: We
may choose coordinates $z_0,\ldots,z_4$ for $\IP^4$ in such a way that 
$\tilde S$ is the vanishing locus of the minors of the matrix $\left(\begin{smallmatrix}
z_1&z_2&z_3\\z_2&z_3&z_4\end{smallmatrix}\right)$ and $c=[0:1:0:0:-1]$. Let
the central projection be given by $x_i=z_i$ for $i=0,2,3$ and $x_1=z_1+z_4$,
so that $S=\{g=0\}$ with $g=x_1x_2x_3-x_2^3-x_3^3$. For a generic choice of $[a]\in P$,
the hyperplane $\{a_0z_0+\ldots+a_4z_4=0\}$ produces a curve in $\tilde S$ defined by
the equation $g=0$ and the vanishing of the minors of 
$$\left(\begin{array}{ccc}
a_0x_0+a_4x_1+a_2x_2+a_3x_3&x_2&\sfrac12(a_4-a_1)x_3\\
\sfrac12(a_4-a_1)x_2&x_3&-a_0x_0-a_1x_1-a_2x_2-a_3x_3
\end{array}\right).$$
This fails to give a curve only if $a_0$, $a_1$ and $a_4$ vanish simultaneously, i.e.\
along $T^\perp\subset P$, and is corrected by the blowing-up of $P$ along $T^\perp$.
The identification in $P'$ that produces $P^\dagger$ is in these coordinates given by 
$[0:0:a_2:a_3:a_4]\mapsto [0:2a_2:a_3:\frac12 a_4:0]$, and it is easy to see that
corresponding matrices yield equal subschemes in $S$. We infer:

\begin{proposition}--- $\Hilb^{gtc}(X_8)_{\red}$ is isomorphic to the 
four-dimensional non-normal projective variety $P^\dagger$.\qed
\end{proposition}
 
Similar calculations can be done for the other non-normal surfaces. In fact, for
the proof of the main theorems we only need the dimension estimate $\dim(\Hilb^{gtc}(X_m))
\leq 4$ for $m=6,7,8,9$, and this result can be obtained much simpler without 
studying the Hilbert schemes themselves using Corollary \ref{cor:dimensionestimate}.

\Section{Moduli of Linear Determinantal Representations}
\label{sec:ModuliOfDeterminantalRepresentations}

This section is the technical heart of the paper. There is a close relation
between generalised twisted cubics on a cubic surface and linear determinantal
representations of that surface as we will explain first. This motivates the
construction of various moduli spaces using Geometric Invariant Theory as a basic tool. 

Fix a three-dimensional projective space $\IP(W)$. We will first 
recall a construction of Ellingsrud, Piene and Str\o mme \cite{EllingsrudPieneStromme} 
of the Hilbert scheme $H_0$ of twisted cubics in $\IP(W)$ in terms of 
determinantal nets of quadrics. We will then adapt their method to construct a 
moduli space of determinantal representations of cubic surfaces in $\IP(W)$, and
establish the relation between these two moduli spaces. The main intermediate
result is the construction of a $\IP^2$-fibration for the Hilbert scheme of
generalised twisted cubics for the universal family of integral cubic surfaces
(Theorem \ref{thm:summaryP2fibration}).

Every step in the construction will be equivariant for the action of $\LieGl(W)$ 
and will therefore carry over to the relative situation for the projective bundle
$a:\IP(\kw)\to \IG$ where $\ko_{\IG}^6\to \kw$ is the tautological quotient of rank
$4$ over the Grassmannian variety $\IG=\Grass(\IC^6,4)$. The ground is then 
prepared for passing to the particular case of the family of cubic surfaces over 
$\IG$  defined by the cubic fourfold $Y\subset \IP^5$. 

Beauville's article \cite{BeauvilleHypersurfaces} gives a thorough foundation
to the topic of determinantal and pfaffian hypersurfaces with numerous references 
to both classical and modern treatments of the subject.

\subsection{Linear determinantal representations}
\label{subsec:determinantaldescriptions}

Let $S=\{g=0\}\subset\IP^3=\IP(W)$ be an integral cubic surface and let 
$C\subset S$ be a 
generalised twisted cubic. We saw earlier that the homogeneous ideal $I_C$ of 
$C$ is generated by the minors of a $3\times 2$-matrix $A_0$ with coefficients
in $W\isom\IC^4$ if $C$ is an aCM-curve. As the cubic polynomial $g\in S^3W$
that defines $S$ must be contained in $I_C$, it is a linear combination 
of said minors and hence can be written as the determinant of a 
$3\times 3$-matrix
\begin{equation}
A=\left(A_0\Big|\begin{smallmatrix}*\\{}*\\{}*\end{smallmatrix}\right).
\end{equation}
As any two such representations of $g$ differ by a relation among the minors
of $A_0$, it follows from the resolution \eqref{eq:aCMSequence} that the 
third column is uniquely determined by $A_0$ up to linear combinations of the 
first two columns. Such a matrix $A$ with entries in $W$ and $\det(A)=g$ is 
called a {\sl linear determinantal representation} of $S$ or $g$. Conversely, 
given a linear determinantal representation $A$ of $g$, any choice of a two-dimensional 
subspace in the space generated by the column vectors of $A$ gives a 
$3\times 2$-matrix $A_0'$. We will see in Section \ref{subsec:P2fib} 
that $A_0'$ is always 
sufficiently non-degenerate to define a generalised twisted cubic. 
In this way every generalised twisted cubic of aCM-type sits in a natural 
$\IP^2$-family of such curves on $S$ regardless of the singularity structure of 
$S$ or $C$. 

If on the other hand $C$ is not CM the situation is similar but slightly 
different: the ideal $I_C$ is cut out by $g$ and the minors of a matrix 
$A_0^t=\left(\begin{smallmatrix}0&-x_0&x_1\\x_0&0&-x_2\end{smallmatrix}\right)$. 
This matrix may be completed to a skew-symmetric matrix as follows:
\begin{equation}A=\left(\begin{smallmatrix}
0&x_0\\ 
-x_0&0\\
x_1&-x_2\end{smallmatrix}\Big|\begin{smallmatrix}
-x_1\\x_2\\0
\end{smallmatrix}\right).\end{equation}
Any $A_0'$ with linearly independent vectors from the space of column vectors
of $A$ defines a non-CM curve on $S$ as before. In fact, the $\IP^2$-family
is in this case much easier to see geometrically: Let $p=\{x_0=x_1=x_2=0\}$
denote the point defined by the entries of $A$, necessarily a singular point
of $S$. Then curves in the $\IP^2$-family simply correspond to hyperplane 
sections through the point $p$.

The $\IP^2$-families of generalised twisted cubics that arise in this way
from $3\times 3$-matrices provide a natural explanation for the appearance
of the $\IP^2$-components of $\Hilb^{gtc}(S)$, if $S$ has at most rational double
points, and for the $\IP^2$-fibration $\Hilb^{gtc}(S)\isom \Sym_3(E)\to E$,
if $S$ has a simple-elliptic singularity. We will exploit this idea further
by constructing moduli spaces of determinantal representations in the 
next section.

We end this section by making the connection between the structure
of $\Hilb^{gtc}(S)$ and the set of essentially different determinantal
representations of $S$ if $S$ is of ADE-type. Here two matrices $A$ and $A'$
are said to give equivalent linear determinantal representations if $A$ can
be transformed into $A'$ by row and column operations. 

Let $S$ be a cubic surface with at most rational double points.
According to the previous discussion, essentially different determinantal
representations correspond bijectively to families of generalised twisted
cubics of aCM-type on $S$. We have seen in Theorem \ref{thm:ADESurfaceCase} that 
these are in natural bijection with $W(R_0)$ orbits on $R\setminus R_0$.

This leads to the data in Table 1: For a surface with at most rational
double points the first column gives the Dynkin type of $R_0$ or equivalently,
the configuration of singularities of $S$, the second column the type notation 
used  by Dolgachev ({\cite{Dolgachev}, Ch. 9) and the third column the number of $W(R_0)$-orbits 
on $R\setminus R_0$. The table can easily be computed  with any all purpose 
computer algebra system.
\begin{table}
\begin{tabular}{c|c|r||c|c|r||c|c|r}
$R_0$&Type&\#&$R_0$&Type&\#&$R_0$&Type&\#\\\hline
$\emptyset$&I&72&$4A_1$&XVI&13&$A_1+2A_2$&XVII&6\\
$A_1$&II&50&$2A_1+A_2$&XIII&12&$A_1+A_4$&XIV&4\\
$2A_1$&IV&34&$A_1+A_3$&X&10&$A_5$&XI&4\\
$A_2$&III&30&$2A_2$&IX&12&$D_5$&XV&2\\
$3A_1$&VIII&22&$A_4$&VII&8&$A_1+A_5$&XIX&1\\
$A_1+A_2$&VI&20&$D_4$&XII&6&$3A_2$&XXI&2\\
$A_3$&V&16&$2A_1+A_3$&XVIII&5&$E_6$&XX&0\\
\end{tabular}\\[1ex]
Table 1: Numbers of inequivalent linear determinantal representations of cubic surfaces of 
given singularity type.\\[2ex]
\end{table}

Here are two examples:

\begin{example} ($3A_2$ singularities) --- Let $p_0,p_1,p_2\in \IP^2$ denote 
the points corresponding to the standard basis in $\IC^3$. Consider the linear 
system of cubics through all three points that are tangent at $p_i$ to the line 
$p_ip_{i+1}$ (indices taken mod $3$). A basis for this linear system is 
$z_0=x_0x_1^2$, $z_1=x_1x_2^2$, $z_2=x_2x_0^2$ and $z_3=x_0x_1x_2$. The image of 
the rational map $\IP^2\dashrightarrow\IP^3$ is the cubic surface $S$ with the 
equation $f=z_0z_1z_2-z_3^3=0$. It has three $A_2$-singularities at the points 
$q_0=[1:0:0:0]$, $q_1=[0:1:0:0]$ and $q_2=[0:0:1:0]$. The reduced Hilbert scheme 
$\Hilb^{gtc}(S)_{\red}$
consists of five copies of $\IP^2$. Three of them are given by the linear
systems $|\ko_S(-q_i)|$, $i=0,1,2$, and correspond to non-CM curves with
an embedded point at $q_i$. The remaining two components correspond to the 2 orbits
listed in the table above. Representatives of these orbits are obtained by 
taking the strict transforms $L$ and $Q$ of a general line $L'$ and a general 
quadric $Q'$ through $p_0$, $p_1$ and $p_2$. To be explicit, take $L'=\{x_0+x_1+x_2=0\}$
and its Cremona transform $Q'=\{x_0x_1+x_1x_2+x_2x_0=0\}$. The corresponding ideals
then are $I_{L}=(z_0(z_2+z_3)+z_3^2, z_1(z_0+z_3)+z_3^2, z_2(z_1+z_3)+z_3^2)$
and $I_{Q}=(z_0(z_1+z_3)+z_3^2,z_1(z_2+z_3)+z_3^2, z_2(z_0+z_3)+z_3^2)$ and differ
only by the choice of a cyclic order of the variables $z_0$, $z_1$ and $z_2$. 
Both $L$ and $Q$ are smooth twisted cubics that pass through all three singularities.
They lead to the following two essentially different determinantal representations
of the polynomial $f$:
$$f=\det\left(\begin{smallmatrix}0&-z_3&z_0\\z_1&0&-z_3\\-z_3&z_2&0\end{smallmatrix}\right)
=\det\left(\begin{smallmatrix}0&-z_3&z_0\\z_2&0&-z_3\\-z_3&z_1&0\end{smallmatrix}\right)
$$
\end{example}

\begin{example} ($4A_1$ singularities) --- Let $\ell_0,\ell_1,\ell_2,\ell_3$ be 
linear forms in three variables that define four lines in $\IP^2$ in general 
position (i.e.\ no three pass through one point) and such that $\sum_i\ell_i=0$. 
The linear system of cubics through the six intersection points has a basis 
consisting of monomials $z_i=\prod_{j\neq i} \ell_j$ for $i=0,\ldots,3$. The 
image of the induced rational map $\IP^2\dashrightarrow\IP^3$ is a cubic surface 
$S$ with the equation 
$$f=z_1z_2z_3+z_0z_2z_3+z_0z_1z_3+z_0z_1z_2$$ 
and with four 
$A_1$-singularities that result from the contraction of the four lines. An 
explicit calculation shows that there are 17 root orbits of different lengths. 
They correspond to families of twisted cubics on $S$ as follows: the transform 
$H$ of a general line in $\IP^2$ gives a twisted cubic on $S$ passing through 
all four singularities. It corresponds to the unique orbit of length $16$ and 
yields the following determinantal representation.
$$f=\det\left(\begin{smallmatrix}
      0& z_0 + z_3& z_0\\    z_1 + z_2& 0&z_1\\z_2&z_3& 0
         \end{smallmatrix}\right)$$
Despite the apparent asymmetry the matrix is in fact symmetric with respect to
all variables up to row and column operations. Now there are 16 possible choices
of non-collinear triples out of the 6 intersection points of the four lines. For each 
triple take a general smooth conic through these points. 
There are four triples that form the {\sl vertices of a triangle} of lines. 
These yield plane curves in $S$ that pass twice through the singularity 
corresponding to the line not in the triangle: the associated generalised twisted 
cubics are non-CM and do not lead to linear determinantal representations.
They account for four orbits of effective roots of length 2. The remaining 12 
triples of points yield families of twisted cubics that pass through any two 
out of the four singularities. These families account for the remaining 
12 inequivalent linear determinantal representations and correspond to
root orbits of length 4.
\end{example}

\subsection{Kronecker modules I: twisted cubics}\label{sec:KroneckerModulesI}

Let the group $\LieGl_3\times \LieGl_2$ act on $U_0:=\Hom(\IC^2,\IC^3\tensor W)$,
with $W\isom\IC^4$, by 
\begin{equation}(g,h)\cdot A_0= (g\tensor \id_W) A_0 h^{-1}.\end{equation}
We will think of homomorphisms $A_0\in U_0$ as $3\times 2$-matrices with values
in $W$ and write simply $A_0\mapsto gA_0h^{-1}$ for the action. \BB
The diagonal subgroup $\Delta_0=\{(tI_3,tI_2)\;|\; t\in \IC^*\}$ acts
trivially, so that the action factors through the reductive group $G_0=\LieGl_3
\times\LieGl_2/\Delta_0$. We are interested in the invariant theoretic 
quotient $U_0^{ss}\GIT G_0$. For an introduction to geometric invariant theory
see any of the standard texts by Mumford and Fogarty \cite{Mumford} or Newstead \cite{Newstead}.
In the given context, the conditions for $A_0$ to be semistable resp.\ stable 
were worked out by Ellingsrud, Piene and Str\o mme. The general case for 
arbitrary $W$ and arbitrary ranks of the general linear groups was treated by 
Drezet \cite{Drezet} and Hulek \cite{Hulek}. We refer to these papers for proofs 
of the following lemma and of Lemma \ref{lem:Ustability}.

\begin{lemma}\label{lem:Unoughtstability} --- A matrix $A_0\in U_0$ is semistable 
if and only if it does not lie in the $G_0$-orbit of a matrix of the form
\begin{equation}
\left(\begin{smallmatrix}*&*\\{}0&*\\0&*\end{smallmatrix}\right)
\quad\text{ or }\quad \left(\begin{smallmatrix}*&*\\{}*&*\\0&0\end{smallmatrix}\right)
\end{equation}
In this case, $A_0$ is automatically stable. The isotropy subgroup of any stable
matrix is trivial.\qed
\end{lemma}

Let $U_0^{s}=U_0^{ss}\subset U_0$ denote the open subset of stable points. Then
$$X_0:=U_0^{s}\GIT G_0$$ 
is a 12-dimensional smooth projective variety, and the
quotient map 
$$q_0:U_0^{ss}\to X_0$$ 
is a principal $G_0$-bundle. There is 
a universal family of maps $a_0:F_0\to E_0\tensor W$, where $F_0$ and $E_0$ 
are vector bundles of rank $2$ and $3$, respectively, on $X_0$ with $\det(F_0)
=\det(E_0)$. Moreover, $\Lambda^2a_0:E_0\to S^2W$ is an injective bundle map 
and defines a closed embedding $X_0\to \Grass(3, S^2W)$ into the Grassmannian 
of nets of quadrics on $\IP(W)$, see \cite{EllingsrudPieneStromme}. Let $I_0
\subset \IP(W)\times \IP(W^*)$ denote the incidence variety of all pairs $(p,V)$
consisting of a point $p=\{x_0=x_1=x_2\}$ on a hyperplane 
$V=\{x_0=0\}$. Sending $(p,V)$ to the 
net $(x_0^2,x_0x_1,x_0x_2)$ defines a map $I_0\to \Grass(3,S^2W)$. Ellingsrud, 
Piene and Str\o mme show that this map is a closed immersion, that it factors 
through $X_0$, and that the Hilbert scheme $H_0$ of twisted cubics on $\IP^3$ 
is isomorphic to the blow-up of $X_0$ along $I_0$. Finally, under the isomorphism 
$H_0\isom \Bl_{I_0}(X_0)$, the divisor $J_0=H_0\cap H_1$  is identified with 
the exceptional divisor. We let $\pi_0:H_0\to X_0$ denote the contraction of $J_0$.
$$
\begin{array}{ccc}
J_0&\lra&H_0\\
\big\downarrow&&\phantom{\scriptstyle \pi_0}\big\downarrow{\scriptstyle \pi_0}\\
I_0&\lra&X_0 
\end{array}
$$

\subsection{Kronecker modules II: determinantal representations}
\label{sec:KroneckerModulesII}

The reductive group $G=\LieGl_3\times \LieGl_3/\Delta$, with 
$\Delta=\{(tI_3,tI_3)\;|\; t\in \IC^*\}$, acts on the affine space 
$$U=\Hom(\IC^3,\IC^3\tensor W)$$
with the analogous action by $(g,h).A:=gAh^{-1}$. 
In contrast to the case of $3\times 2$-matrices the notions of stability and 
semistability differ here. Again, this is a special case of a more general result 
of Drezet and Hulek. 

\begin{lemma}\label{lem:Ustability} --- A matrix $A\in U$ is semistable if it does 
not lie in the $G$-orbit of a matrix of the form
\begin{equation}\left(\begin{smallmatrix}0&*&*\\{}0&*&*\\0&*&*\end{smallmatrix}\right)
\quad\text{ or }\quad 
\left(\begin{smallmatrix}*&*&*\\{}0&0&*\\0&0&*\end{smallmatrix}\right)
\quad\text{ or }\quad 
\left(\begin{smallmatrix}*&*&*\\{}*&*&*\\0&0&0\end{smallmatrix}\right)\end{equation}
and is stable if it does not lie in the $G$-orbit of a matrix of the form
\begin{equation}\left(\begin{smallmatrix}*&*&*\\{}0&*&*\\0&*&*\end{smallmatrix}\right)
\quad\text{ or }\quad 
\left(\begin{smallmatrix}*&*&*\\{}*&*&*\\0&0&*\end{smallmatrix}\right)\end{equation}
The isotropy subgroup of any stable matrix is trivial.\qed
\end{lemma}

Consequently, the quotient 
$$X:=U^{ss}\GIT G$$ 
is an irreducible normal projective variety of dimension $\dim X=\dim U-\dim G=19$. 
The stable part $X^s=U^s\GIT G$ is a smooth dense open subset, and the quotient 
$$q^s:U^s\to X^s$$ 
is a principal $G$-bundle. The character group of $G$ is generated by $\chi:G\to \IC^*$, 
$\chi(g,h)=\det(g)/\det(h)$, and the trivial line bundle $\ko_U(\chi)$, endowed 
with the $G$-linearisation defined by $\chi$, descends to the ample generator 
$L_X$ of $\Pic(X)$. 

The tautological homomorphism $a_U: \ko_U^3\to \ko_U^3\tensor W$ induces a map
$\det(a_U):\ko_U(-\chi)\to \ko_U\tensor S^3W$ that descends to a homomorphism 
$\det: L_X^{-1}\to \ko_X\tensor S^3W$, which in turn induces a rational map
$\det:X\dashrightarrow \IP(S^3W^*)$. We need to understand the degeneracy
locus of this map.

\begin{proposition}\label{prop:determinantalrepresentation}--- Let $A\in U^{ss}$ 
be a semistable matrix and consider its determinant $\det(A)\in S^3W$.
\begin{enumerate}
\item If $A$ is semistable but not stable, then $\det(A)$ is a non-zero reducible
polynomial.
\item If $\det(A)=0$, then $A$ is stable and is conjugate under the $G$-action
to a skew-symmetric matrix. 
\end{enumerate}
\end{proposition}

\begin{lemma}\label{lem:Bmatrix}--- Let $B$ be a matrix with values in a 
polynomial ring over a field. If $\rk(B)\leq 1$, i.e.\ if all 
$2\times 2$-minors of $B$ vanish, there are vectors $u$ and $v$ with values in 
the polynomial ring such that $B=vu^t$. If all entries of $B$ are homogeneous 
of the same degree then the same is true for both $u$ and $v$. 
\end{lemma}

\begin{proof} We may assume that $B$ has no zero columns. Extracting from each
column its greatest common divisor we may further assume that each column 
consists of coprime entries. As all columns are proportional over the function
field we find for each pair of column vectors $B_i$ and $B_j$ coprime polynomials
$g_i$ and $g_j$ such that $g_jB_i=g_iB_j$. As $g_i$ and $g_j$ are coprime, $g_i$
must divide every entry of $B_i$. Hence $g_i$ is unit, and for symmetry reasons
$g_j$ is as well. Therefore all columns of $B$ are proportional over the ground 
field. The last assertion follows easily.
\end{proof}

\begin{proofof}{Proposition \ref{prop:determinantalrepresentation}}~\\
1. Assume first that $A$ is semistable but not 
stable. Replacing $A$ by another matrix from its orbit we may assume that 
\begin{equation}A=\left(\begin{smallmatrix}*&*&*\\{}0&*&*\\0&*&*\end{smallmatrix}\right)
\quad\text{ or }\quad 
A=\left(\begin{smallmatrix}*&*&*\\{}*&*&*\\0&0&*\end{smallmatrix}\right).\end{equation}
It is clear that $\det(A)$ factors into a linear and a quadric polynomial in $S^*W$.
If $\det(A)=0$, either the linear or the quadratic factor must vanish. If the linear
factor vanishes $A$ has a trivial row or column, which contradicts its semistability.
If the quadratic polynomial vanishes, the lower right respectively upper left 
$2\times 2$-block $B$ satisfies $\det(B)=0$. According to Lemma \ref{lem:Bmatrix},
appropriate row or column operations will eliminate a row or column of $B$. This
contradicts again the semistability of $A$. 
 
2. Let $A$ be a stable matrix with $\det(A)=0$ \BB and let $C=\adj(A)\in (S^2W)^{3\times 3}$ denote 
its adjugate matrix. So $C_{ij}=(-1)^{i+j}\det(A^{ji})$ where $A^{ji}$ is
the matrix obtained from $A$ by erasing the $j$-th row and the $i$-th column. If
$\det(A^{ji})$ were $0$, the rows or columns of $A^{ji}$ would be $\IC$-linearly
dependent according to Lemma \ref{lem:Bmatrix}. Row or column operations applied
to $A$ would produce a row or a column with at least two zeros, contradicting the
stability of $A$. This shows that all entries of $C$ are non-zero, and this holds
even after arbitrary row and column operations on $C$, since such operations
correspond to column resp. row operations on $A$. In particular, all columns 
and all rows of $C$ contain $\IC$-linearly independent entries. Since $\adj(C)=
\det(A)A=0$, one has $\rk(C)\leq 1$. By Lemma \ref{lem:Bmatrix}, there are 
homogeneous column vectors $u,v\in S^*W$ such that $C=uv^t$. Since the entries 
of the rows and columns of $C$ are $\IC$-linearly independent, $u$ and $v$ must 
have entries of degree $1$, and these must be linearly independent for each vector. 
In an appropriate basis $x_0,x_1,x_2,x_3$ of $W$ we may write $u=(x_2\, x_1\, x_0)^t$.
Since the entries of $u$ form a regular sequence their syzygy module is given by 
the Koszul matrix 
$K=\left(\begin{smallmatrix}0&x_0&-x_1\\-x_0&0&x_2\\x_1&-x_2&0\end{smallmatrix}\right)$.
Since $AC=0$ implies $Au=0$, it follows that $A=MK$ for some $M\in\IC^{3\times 3}$.
Finally, since the columns of $A$ are $\IC$-linearly independent because of the
stability of $A$, the transformation matrix $M$ must be invertible, and  $A\sim_G K$ 
as claimed.
\end{proofof}

The proposition allows for a simple stability criterion in terms of the determinant:

\begin{corollary}\label{cor:stabilitycriterion} --- 
For any $A\in U$ the following holds:
\begin{enumerate}
\item If $\det(A)\neq 0$, then $A$ is semistable.
\item If $\det(A)$ is irreducible, then $A$ is stable.
\item If $A$ is stable, then either $\det(A)\neq 0$ or $A$ is in the $G$-orbit of
a skew-symmetric matrix.
\end{enumerate}
\qed
\end{corollary}

We continue the discussion of the rational map $\det:X\dashrightarrow \IP(S^3W^*)$.
The following commutative diagram is inserted here as an optical guide through the 
following arguments. The notation will be introduced step by step.
\begin{equation}
\begin{array}{c}
\xymatrix{
\Hom'(\IC^3,W)\ar[r]^{\;\;\;\isom}\ar[rd]_{\GIT \LieGl_3}&T^{ss}\ar[d]^{\GIT\Gamma}
\ar[r]&U^{ss}\text{\makebox[0mm][l]{$\subset\Hom(\IC^3,\IC^3\tensor W)$}}\ar[d]^{\GIT G}\\
&\IP(W)\ar[r]&X\ar@{-->}[r]^{\det\;\;\;\;\;\;}&\IP(S^3W^*)\\
&\IP(N')=J\ar[u]^\sigma\ar[r]&H\ar[u]^\sigma\ar[ur]^{\delta}
}
\end{array}
\end{equation}
Consider the splitting $U=V\oplus T$ into the subspaces $V=\{a\in U\;|\;a^t=a\}$
of symmetric and $T=\{a\in U\;|\; a^t=-a\}$ of skew-symmetric matrices. According 
to Proposition \ref{prop:determinantalrepresentation}, the smooth closed subset 
$$T^{ss}:=T\cap U^{ss}\subset U^{ss}$$
is in fact contained in the open subset $U^s$ of stable points, and its $G$-orbit 
$G.T^{ss}$ is the vanishing locus of the determinant $\det(a_U):\ko_{U^{ss}}(-\chi)
\to \ko_{U^{ss}}\tensor S^3W$. An element $A\in T^{ss}$ is mapped back to $T^{ss}$ by 
$[g,h]\in G$ if and only if $(gAh^{-1})^t=-gAh^{-1}$. This is equivalent to 
saying that $[h^tg, g^th]$ is a stabiliser of $A$. Hence $h=\lambda (g^t)^{-1}$ 
for some $\lambda\in \IC^*$. In fact, changing $h$ and $g$ by an appropriate 
scalar, we get $[g,h]=[\gamma,(\gamma^{t})^{-1}]$ for some $\gamma\in \LieGl_3$, 
well-defined up to a sign $\pm 1$. We conclude that 
$$T^{ss}\GIT\Gamma=G.T^{ss}\GIT G\;\subset\; U^{ss}\GIT G=X,$$ 
where $\Gamma:=\LieGl_3/\pm I$ acts freely on $T^{ss}$ via $\gamma.A=\gamma A\gamma^t$.
Any deformation $a\in U$ of $A\in T^{ss}$ can be split into its symmetric and its
skew-symmetric part. The skew-symmetric part gives a tangent vector to $T^{ss}$ at 
$A$. Among the symmetric deformations those of the form $uA-Au^t$, $u\in \Liegl_3
\isom \Lie(\Gamma)$, are tangent to the $G$-orbit of $A$. The bundle homomorphism
\begin{equation}\label{eq:bundlehomomorphismrho}
\rho:\Liegl_3\tensor \ko_{T^{ss}}\to V\tensor \ko_{T^{ss}},
\quad (A,u)\mapsto (A,uA-Au^t),
\end{equation}
is equivariant with respect to the natural action of $\gamma\in \Gamma$ given by
$\gamma.u=\gamma u\gamma^{-1}$ and $\gamma.a=\gamma a\gamma^t$ and has constant 
rank $8$. The cokernel of $\rho$ therefore has rank 16 and is isomorphic to the 
restriction to $T^{ss}$ of the normal bundle of $G.T^{ss}$ in $U^{ss}$. It descends 
to the normal bundle of $T^{ss}\GIT\Gamma$ in $X$. 

We can look at $T^{ss}$ in a different way that will lead to an isomorphism
$T^{ss}\GIT \Gamma\isom \IP(W)$ and to an identification of its normal bundle:
Let $\Hom'(\IC^3,W)$ denote the open subset of injective homomorphisms 
$v:\IC^3\to W$. The group $\LieGl_3$ acts naturally on $\IC^3$, and we consider 
the induced action on $\Hom'(\IC^3,W)$ given by $g.v:=v\circ g^{-1}$. The 
projection $\Hom'(\IC^3,W)\to \IP(W)$ is a principal fibre bundle with respect 
to this action. The isomorphism
$$\tau:\Hom'(\IC^3,W)\to T^{ss},\quad v\mapsto \left(\begin{matrix}
0&v(e_3)&-v(e_2)\\
-v(e_3)&0&v(e_1)\\
v(e_2)&-v(e_1)&0\end{matrix}\right)$$
is equivariant for the group isomorphism 
$$\LieGl_3\to \Gamma=\LieGl_3/\pm I_3,\quad h\mapsto \frac{h}{\sqrt{\det(h)}}.$$
We conclude that $\IP(W)=\Hom'(\IC^3,W)\GIT\LieGl_3\isom T^{ss}\GIT\Gamma$.
The pull-back of the bundle homomorphism $\rho$ in \eqref{eq:bundlehomomorphismrho}
to $\Hom'(\IC^3,W)$ via $\tau$ is a homomorphism
$$\hat\rho: \Hom'(\IC^3,W)\times \Liegl_3\lra \Hom'(\IC^3,W)\times V,
\quad (v,u)\mapsto (v, u\tau(v)-\tau(v)u^t),$$
that is $\LieGl_3$-equivariant with respect to the adjoint representations on 
$\Liegl_3$ and the representation 
$$h.a=\frac{h}{\sqrt{\det(h)}}a\frac{h^t}{\sqrt{\det(h)}}=\frac1{\det(h)}
hah^{t}$$
on $V$. 
The trivial bundle $\Hom'(\IC^3,W)\times \IC^3$ descends to the kernel $K$ in the
tautological sequence $0\to K\to W\tensor \ko_{\IP(W)}\to \ko_{\IP(W)}(1)\to 0$
on $\IP(W)$. Accordingly, the homomorphism $\hat\rho$ descends to a bundle 
homomorphism
$$\tilde \rho:\End(K)\to  S^2K\tensor W\tensor\det(K)^{-1}$$
on $\IP(W)$. Rewriting the first sheaf as $\End(K)=K\tensor K^*=
K\tensor \Lambda^2K\tensor \det(K)^{-1}$, this bundle map is explicitly given by 
$w\tensor w'\wedge w''\tensor \mu\mapsto (ww'\tensor w''-ww''\tensor w')\tensor\mu$. 
In particular, the cokernel of $\tilde \rho$ is isomorphic to $N\tensor \det(K)^{-1}$, 
where 
$$N:=\im (S^2K\tensor_\IC W\to \ko_{\IP(W)}\tensor_\IC S^3W)$$ 
is the image of the natural multiplication map. From this we conclude:

\begin{proposition}--- The morphism $i:\IP(W)\isom T^{ss}\GIT\Gamma\hookrightarrow 
X$ constructed above is an isomorphism onto the indeterminacy locus of the rational 
map $\det:X\dashrightarrow \IP(S^3W^*)$. The normal bundle of $\IP(W)$ in $X$ is isomorphic 
to $N\tensor \det(K)^{-1}$, and $i^*(L_X)\isom \det(K)^{-1}\isom 
\det(W)^{-1}\tensor \ko_{\IP(W)}(1)$.
\end{proposition}

\begin{proof} Only the last statement has not yet been shown. In fact, the 
composite character $\chi':\LieGl_3\xra{\;\isom\;} \Gamma\hookrightarrow 
G\xra{\;\chi\;}\IC^*$ is given by $\chi'(h)=\det(\sfrac{h}{\sqrt{\det(h)}})^2=
\det(h)^{-1}$. This implies $i^*L_X\isom \det(K)^{-1}$. It follows from the 
exactness of the tautological sequence 
$0\to K\to \ko_{\IP(W)}\tensor W\to \ko_{\IP(W)}(1)\to 0$
that $\det(K)^{-1}\isom \det(W)^{-1}\tensor_\IC\ko_{\IP(W)}(1)$. 
\end{proof}

The one-dimensional vector space $\det(W)$ appears in the proposition in order to keep
all statements equivariant for the natural action of $\LieGl(W)$.
%
Let \BB
$$
\begin{array}{ccc}
J&\lra&H\\
{\scriptstyle \sigma}\Big\downarrow\phantom{\scriptstyle \sigma}&&
\phantom{\scriptstyle \sigma}\Big\downarrow{\scriptstyle \sigma}\\
\IP(W)&\lra &X
\end{array}
$$ 
denote the blow-up of $X$ along $\IP(W)$ with exceptional divisor $J$. 
According to the previous proposition $J=\IP(N')$, 
where $N':=(N\tensor\det(K)^{-1})^*$. 
Note that the fibre of $\sigma: J\to \IP(W)$ over a point $p$ is exactly the 
$\IP^{15}$-family of cubic surfaces that are singular at $p$. \BB
The Picard group of $H$ is generated by 
$\sigma^*L_X$ and $\ko_H(J)$.

\begin{proposition}--- The rational map $\det:X\dashrightarrow\IP(S^3W^*)$ extends to
a well-defined morphism 
$$\delta:H\to \IP(S^3W^*).$$
Moreover, there are bundle isomorphisms
$$\ko_H(J)|_J\isom\ko_{N'}(-1)\quad\text{and}\quad 
\delta^*\ko_{\IP(S^3W^*)}(1)\isom \sigma^* L_X\tensor \ko_H(-J).$$
\end{proposition}

\bigskip
In view of this proposition we may call $H$ the {\sl universal linear determinantal
representation}.

\begin{proof} Let $p\in \IP(W)$ be defined by the vanishing of the linear
forms $x_0,x_1,x_2\in W$. Its image in $X$ is represented by the skew-symmetric
matrix $A=\left(\begin{smallmatrix}0&x_0&-x_1\\-x_0&0&x_2\\x_1&-x_2&0
\end{smallmatrix}\right)\in T^{ss}$. The $16$-dimensional vector space 
$$N_0:=\{a\in U\;|\; a=a^t\}/\{uA-Au^t\;|\; u\in\Liegl_3\}$$ represents
a slice transversal to the $G$-orbit through $A$, as we have seen before. 
The differential of $\det:U\to S^3W$ restricted to $A+N_0$ at $A$ equals
\begin{equation}
(D_A\det)(a)=\tr(a\adj(A))
=\left(\begin{smallmatrix}x_0\\x_1\\x_2\end{smallmatrix}\right)^ta
 \left(\begin{smallmatrix}x_0\\x_1\\x_2\end{smallmatrix}\right).
 \end{equation}
An explicit calculation now shows that $D_A\det:N_0\to S^3W$ is injective. 
This implies that $\det:X\setminus \IP(W)=H\setminus J\to \IP(S^3W^*)$ extends 
to a morphism $\delta:H\to \IP(S^3W^*)$. The restriction $\delta|_J:J=\IP(N')
\to \IP(S^3W^*)$ is induced by the bundle epimorphisms
$\ko_{\IP(N')}\tensor_\IC S^3W^*\twoheadrightarrow \sigma^*N^*\twoheadrightarrow
\ko_{N'}(1)\tensor\sigma^*\det(K)^{-1},$ 
so that 
$$
\delta^*\ko_{\IP(S^3W^*)}(1)|_{\IP(N')}=\ko_{N'}(1)\tensor \sigma^*\det(K)^{-1}.
$$
There are integers $m,m'$ such that $\delta^*\ko_{\IP(S^3W^*)}(1)=
\sigma^*L_X^m\tensor\ko_H(J)^{m'}$. The restriction to $J$ becomes
$$\delta^*\ko_{\IP(S^3W^*)}(1)|_J=\sigma^*(L_X^m|_{\IP(W)})\tensor
\ko_H(J)|_J^{m'}=\det(K)^{-m}\tensor \ko_{N'}(-m').$$
Comparison of the two expressions
for $\delta^*\ko_{\IP(S^3W^*)}(1)|_J$ shows $m=1$ and $m'=-1$.
\end{proof}

\begin{corollary} \label{cor:relativeAmpleness} --- 
The line bundle $\ko_H(J)$ is ample relative $\delta:H\to 
\IP(S^3W^*)$. 
\end{corollary}

\begin{proof} Let $F\subset H$ be a subvariety of a fibre of $\delta$. 
Then 
$$\ko_F\isom \delta^*\ko_{\IP(S^3W^*)}(1)|_F\isom \sigma^*L_X|F\tensor 
\ko_H(-J)|_F,$$
so that $\ko_H(J)|_F\isom \sigma^*L_X|_F$. Since $\delta$ is an embedding on 
fibres of $\sigma$, the variety $F$ projects isomorphically into $X$. Hence
$\sigma^*L_X|_F$ is ample.
\end{proof}

\begin{corollary}\label{cor:dimensionestimate}--- For any cubic surface 
$S\subset\IP(W)$ the $\delta$-fibre  over the corresponding point 
$[S]\in \IP(S^3W^*)$ is finite if $S$ has at most ADE-singularities and 
satisfies the estimate
$$\dim\,\delta^{-1}([S])\leq \dim \,\Sing(S)+1,$$
otherwise. 
\end{corollary}

\begin{proof} The case of surfaces with ADE-singularitites was treated in 
Section \ref{sec:TwistedCubicsOnCubicSurfaces}. Otherwise,
a point in $J$ encodes a point $p\in \IP(W)$ together with a cubic 
surface $S$ that is singular at $p$. Hence $J\cap \delta^{-1}([S])$ is isomorphic 
to the singular locus of $S$ through projection to $\IP(W)$. Since $J$ is an 
effective Cartier divisor that is ample relative $\delta$, the intersection with 
every irreducible component of $\delta^{-1}([S])$ of positive dimension is 
non-empty and of codimension $\leq 1$ in this component. This implies the asserted 
inequality.
\end{proof}

\subsection{The $\IP^2$-fibration for the universal family of cubic surfaces}
\label{subsec:P2fib}

Let 
$$R\subset H_0\times\IP(S^3W^*)$$
denote the incidence variety of all points
$([C],[S])$ such that the generalised twisted cubic $C$ is contained in the
cubic surface $S$. Of the two projections $\alpha:R\to H_0$ and $\beta:R\to 
\IP(S^3W^*)$ the first is a $\IP^9$-bundle by \cite{EllingsrudStromme}, Cor.~2.4,
so that $R$ is smooth and of dimension $21$. We have arrived at the following
set-up:
\begin{equation}
\begin{array}{c}
\xymatrix{
&R\ar[ld]_{\alpha}^{\IP^9}\ar[rd]^{\beta}&&H\ar[ld]_{\delta}\ar[rd]^{\sigma}\\
H_0&&\IP(S^3W^*)&&\;X\;\ar@{-->}[ll]_{\det}.}
\end{array}
\end{equation}
Consider the open subset $\IP(S^3W^*)^{\ganz}\subset\IP(S^3W^*)$ of integral 
surfaces and the corresponding open subsets 
$$H^{\ganz}=\delta^{-1}(\IP(S^3W^*)^{\ganz})\quad\text{ and }\quad 
R^{\ganz}=\beta^{-1}(\IP(S^3W)^{\ganz}).$$ 
By part (1) of Proposition \ref{prop:determinantalrepresentation}, one has
$H^{\ganz}\subset H^s\subset H$, where $H^s=\sigma^{-1}(X^s)$.

For any matrix $A\in U$ let $\res(A)\in U_0$ denote the submatrix consisting
of its first
two columns. A comparison of the Lemmas \ref{lem:Ustability} and 
\ref{lem:Unoughtstability} shows immediately, that $\res$ restricts to a map
$\res:U^s\to U_0^s$. Let $P'\subset \LieGl_3$ denote the parabolic subgroup of
elements that stabilise the subspace $\IC^2\times \{0\}\subset\IC^3$. The 
parabolic subgroup $P=(\LieGl_3\times P')/\IC^*\subset G$ has a natural
projection $\gamma:P\to G_0$ through its Levi factor, and $\res:U^s\to U_0^s$
is equivariant with respect to this group homomorphism, i.e.\ $\gamma(p).\res(A)=
\res(p.A)$ for all $A\in U^s$ and $p\in P$.

Since $q^s:U^s\to X^s$ is a principal $G$-bundle, it factors through maps
\begin{equation}U^s\xra{\;q_P\;}U^s/P\xra{\;a_P\;} U^s\GIT G=X^s,\end{equation}
where $a_P$ is an \'etale locally trivial fibre bundle with fibres isomorphic 
to $G/P\isom \IP^2$. As $\res$ is $\gamma$-equivariant it descends to a 
morphism $\overline\res:U^s/P\to X_0=U_0^s/G_0$. This provides us with
morphisms 
\begin{equation}
X_0\xla{\;\;\overline\res\;\;}U^s/P\xra{\;\; a_P\;\;}X^s.
\end{equation}
Let $\sigma_Q:Q\to U^s/P$ denote the blow-up along $a_P^{-1}(I)$. By the universal
property of the blow up, there is a natural morphism $a_Q:Q\to H^s$, which
is again a $\IP^2$-bundle.
$$\begin{array}{ccc}
Q&\xra{\sigma_Q}&U^s/P\\
{a_Q}\Big\downarrow\phantom{a_Q}&&\phantom{a_P}\Big\downarrow{a_P}\\ 
H^s&\xra{\;\;\sigma\;\;}&X^s
\end{array}
$$
Let $Q^{\ganz}=a_Q^{-1}(H^{\ganz})$. 

\begin{proposition}--- $R^{\ganz}\isom Q^{\ganz}$ as schemes over $X_0\times
\IP(S^3W^*)^{\ganz}$
\end{proposition}

\begin{proof} $Q^{\ganz}$ parameterises via the composite morphism $Q^{\ganz}
\to H^{\ganz}\to \IP(S^3W^*)$ a family of cubic surfaces $S_q=\{g_q=0\}$, 
$q\in Q^{\ganz}$, and via the composite morphism $Q^{\ganz}\to U^s/P\to X_0$ a 
family of determinantal nets of quadrics $(Q^{(1)}_q, Q^{(2)}_q,Q^{(3)}_q)$, 
$q\in Q^{\ganz}$, in such a way that either the ideal 
$I_q:=(Q^{(1)}_q, Q^{(2)}_q,Q^{(3)}_q)$ 
defines an aCM generalised twisted cubic on the surface $S_q$,
or $I_q$ is the ideal of a hyperplane with an embedded point on $S_q$. But 
in both cases the ideal $I_q':=I_q+(g_q)$ defines a generalised twisted cubic
$C_q$ on $S_q$. As the base scheme $Q^{\ganz}$ of this family is reduced and 
the Hilbert polynomial of the family of curves $C_q$ is constant, this family 
is flat. 
Since $R$ is the moduli space 
of pairs $(C\subset S)$ of a generalised twisted cubic on a cubic surface, there 
is classifying morphism $\psi:Q^{\ganz}\to R$ whose image is obviously contained 
in $R^{\ganz}$. As both $Q^{\ganz}$ and $R^{\ganz}$ are smooth it suffices to
show that $\psi$ is bijective. 

Let $([A],g)$ be a point in $Q^{\ganz}$. We need to show that $A$ can be 
reconstructed up to the action of $P$ from $([A_0],g)$ where $A_0=\res(A)$.
If $A_0$ defines an aCM-curve, it follows from the presentation \eqref{eq:aCMSequence} 
that any extension of $A_0$ to a matrix $B$ with $\det(B)=g$ and $\res(B)=A_0$ 
is unique up to adding multiples of the first two columns to the last. 
But this is exactly the way that $P$ acts on the columns of $A$. If on the other 
hand $A_0$ (together with $g$) defines a non-CM curve, the point $[A_0]$ belongs 
to $I_0$, and the determinant of any $B$ with $\res(B)=A_0$ will split off a 
linear factor. As $[B]$ is required to lie in $Q^{\ganz}$ this is only possible 
when $\det(B)=0$ according to part (1) of Proposition 
\ref{prop:determinantalrepresentation}. By part (2) of the same proposition
it follows again that $B$ is in the $P$-orbit of $A$. 
This proves the injectivity of $\psi$.

Assume finally that a point $n\in R^{\ganz}$ be given. It determines and is 
determined by a pair $([A_0],g)$. If $[A_0]\in I_0$, the existence of a stable
matrix $A$ with $\res(A)=A_0$ is clear. If $[A_0]\not\in I_0$, there is a unique
matrix $A\in U$ up to column transformations with $\res(A)=A_0$ and $\det(A)=g$.
Since $g$ is non-zero and irreducible, $A$ is stable. This shows that $\psi$ is 
surjective as well.
\end{proof}

We can summarise the results of this section as follows:

\begin{theorem}\label{thm:summaryP2fibration}--- Let $R^{\ganz}$ denote the 
moduli space of pairs $(C,S)$ of an integral cubic surface $S$ and a 
generalised twisted cubic $C\subset S$ in a fixed three-dimensional projective
space $\IP(W)$. 
\begin{enumerate}
\item The projection $R^{\ganz}\to H_0$ to the first component is a surjective 
smooth morphism whose fibres are open subsets in $\IP^9$. In particular, $R^{\ganz}$
is smooth.
\item The projection 
$R^{\ganz}\to \IP(S^3W^*)^{\ganz}$ is projective and factors as follows:
\begin{equation}\label{eq:P2fibration}
R^{\ganz}\xra{\;a_R\;} H^{\ganz}\xra{\;\delta\;} \IP(S^3W^*)^{\ganz},
\end{equation}
where $a_R$ is a $\IP^2$-bundle and $\delta$ is generically finite.
\end{enumerate}
\qed
\end{theorem}

\Section{Twisted Cubics on $Y$}
\label{sec:NowOnY}

In the previous Section \ref{sec:ModuliOfDeterminantalRepresentations} we
have discussed the geometry of generalised twisted cubics on cubic surfaces
for the universal family of cubic surfaces in a fixed 3-dimensional projective
space $\IP(W)$, the main result being the construction of maps
$$H_0\lla R^{\ganz}\lra H\xra{\;\delta\;} \IP(S^3W^*).$$

The cubic fourfold $Y$ has played no rôle in the discussion so far. 
The intersections of $Y$ with all $3$-spaces in $\IP^5$ form a family of
cubic surfaces parameterised by the Grassmannian $\IG=\Grass(\IC^6,4)$. 
All schemes discussed in the previous section come with a natural $\LieGl(W)$-action,
and all morphisms are $\LieGl(W)$-equivariant. This allows us to generalise all
results to this relative situation over the Grassmannian.

In this section we will construct the morphisms $\Hilb^{gtc}(Y)\to Z'\to Z$
and prove that $Z$ is an 8-dimensional connected symplectic manifold.

\subsection{The family over the Grassmannian}

Let $\IG:=\Grass(\IC^6,4)$ denote as before the Grassmannian of three-dimensional
linear subspaces in $\IP^5$, let $\ko_{\IG}^6\to \kw$ denote the universal quotient
bundle of rank $4$. The projectivisation $\IP(\kw)$ is a partial flag variety and 
comes with two natural projections $a:\IP(\kw)\to \IG$ and 
$q:\IP(\kw)\to \IP^5$. Let
\begin{equation}
0\to K\to a^*\kw\to \ko_a(1)\to 0
\end{equation}
denote the tautological exact sequence. Then
$\det(K)^{-1}=\ko_a(1)\tensor a^*\det(\kw)^{-1}$. Furthermore, let $\IS:=\IP(S^3\kw^*)$ 
denote the space of cubic surfaces in the fibres of $a$, let $\IS^{\ganz}\subset \IS$ 
denote the open subset corresponding to integral surfaces, and let $c:\IS\to \IG$ denote
the  natural projection.

We will build up the following commutative diagram of morphisms step by step:
\begin{equation}\label{eq:BigDiagram} 
\begin{array}{c}
\xymatrix{
\IP(\kn')\ar@{^{(}->}[r]^{j}\ar[d]_{\sigma}&\IH\ar[d]_{\tilde \sigma}\ar[rd]^{\delta}\\
\IP(\kw)\ar[rd]_{a}\ar@{^{(}->}[r]^{i}\ar[d]_{q}&
\IX\ar[d]_{b}\ar@{-->}[r]^{\det\;\;}&{~~~}\IS\ar[ld]_c\\
\IP^5&\IG
}
\end{array}
\end{equation}
Generalising the results of Section \ref{sec:KroneckerModulesII} to the relative 
case we consider the vector bundle $\khom(\IC^3,\IC^3\tensor \kw)$ on $\IG$ and the 
quotient $\IX$ of its open subset of semistable points by the group 
$G=(\LieGl_3\times\LieGl_3)/\IC^*$. The natural projection $b:\IX\to \IG$ is a 
projective morphism and a Zariski locally trivial fibre bundle with fibres isomorphic 
to $X$. There is a canonical embedding $i:\IP(\kw)\to \IX$ of $\IG$-schemes 
such that the normal bundle of $\IP(\kw)$ in $\IX$ is given by 
\begin{equation}\label{eq:N_appears}
\nu_{\IP(\kw)/\IX}\isom \kn\tensor \det(K)^{-1}\isom 
\kn\tensor\ko_a(1)\tensor a^*\det(\kw)^{-1},
\end{equation}
where $\kn$ is the image of the natural multiplication map 
$S^2K\tensor a^*\kw\to a^*S^3\kw$. 
Let $\tilde\sigma:\IH\to \IX$ denote the blow-up of $\IX$ along $\IP(\kw)$. The 
exceptional divisor of $\tilde\sigma$ can be identified with $\IP(\kn')$, where
$\kn':=\nu_{\IP(\kw)/\IX}^*$, and we let $\sigma:\IP(\kn')\to \IP(\kw)$ and 
$j:\IP(\kn')\to \IH$ denote the canonical projection and inclusion, respectively. 
As we have seen in previous sections, the rational map $\det:\IX\dashrightarrow
\IS$ extends to a well-defined morphism $\delta:\IH\to \IS$.

Finally, let $\IH_0\to\IG$ denote the relative Hilbert scheme of generalised twisted
cubics in the fibres of $a:\IP(\kw)\to\IG$, and let $\kr^{\ganz}$ denote the 
moduli space of pairs $(C,S)$ where $S$ is an integral cubic surface in a fibre
of $a$ and $C$ is a generalised twisted cubic in $S$. Generalising Theorem
\ref{thm:summaryP2fibration} to the relative situation over the Grassmannian we
obtain a commutative diagram
\begin{equation}
\begin{array}{c}
\xymatrix{
\IH\ar[d]&\IH^{\ganz}\ar[l]\ar[d]&\kr^{\ganz}\ar[l]_a\ar[d]\\
\IS\ar[dr]&\IS^{\ganz}\ar[l]&\IH_0\ar[dl]\\
&\IG
}
\end{array}
\end{equation}
where $a$ is a $\IP^2$-bundle.

Let $Y\subset\IP^5$ be a smooth cubic hypersurface defined by a polynomial 
$f\in S^3\IC^6$ and assume that $Y$ does not contain a plane. Then $f$ defines 
a nowhere vanishing section in $S^3\kw$ and hence a section $\gamma_f:\IG\to \IS$ 
to the bundle projection $c$. For a point $[\IP(W)]\in\IG$, its image 
$[S]=\gamma_f([\IP(W)])$ is the surface $S=\IP(W)\cap Y$. Since $Y$ does not
contain a plane, $\gamma_f$ takes values in the open subset $\IS^{\ganz}\subset\IS$
of integral surfaces. 

We define a projective scheme $Z'$ with a Cartier divisor $D\subset Z'$ by the
following pull-back diagram
$$
\begin{array}{ccccc}
\IP(\kn')&\hookrightarrow&\IH&\lra&\IS\\
\cup&&\cup&&\cup\\
D&\hookrightarrow&Z'&\lra&\gamma_f(\IG)
\end{array}
$$
As $\gamma_f(\IG)$ is contained in $\IS^{\ganz}$, the scheme $Z'$ is in fact contained
in the open subset $\IH^{\ganz}\subset\IH$. 

\begin{proposition}--- $a^{-1}(Z')\isom \Hilb^{gtc}(Y)$, and $a^{-1}(D)$ is the
closed subset of non-CM curves. 
\end{proposition}

\begin{proof} The natural projection $\Hilb^{gtc}(Y)\to \IG$ lifts both to a closed 
immersion 
$$\Hilb^{gtc}(Y)\to \IH_0$$ 
and to a morphism $\Hilb^{gtc}(Y)\to \IS^{\ganz}$,
sending a curve $C$ with span $\langle C\rangle =\IP(W)$ to the point $[C]
\in \Hilb^{gtc}(\IP(W))\subset\IH_0$ and the point $[\IP(W)\cap Y]$, respectively.
By the definition of $\kr^{\ganz}$, these two maps induce a closed immersion $\Hilb^{gtc}(Y)
\to \kr^{\ganz}$, whose image equals $a^{-1}(Z')$ by Theorem \ref{thm:summaryP2fibration}.
The second assertion follows similarly. 
\end{proof}

We have proved the first part of Theorem \ref{thm:MainTheorem}: the existence of a natural 
$\IP^2$-fibration 
$$\Hilb^{gtc}(Y)\xra{\;\;a\;\;} Z'$$
relative to $\IG$.

\begin{proposition}\label{prop:dimensionboundnonnormal} --- 
Let $Y$ be a smooth cubic fourfold. Then the closure of
the set of points $[\IP(W)]\in \IG$ such that $S=\IP(W)\cap Y$ is a non-normal 
integral surface is at most 4-dimensional. 
\end{proposition}

\begin{proof} Let $L\subset Y=\{f=0\}$ be a line, and let $U\subset\IC^6$ denote 
the four-dimensional space of linear forms that vanish on $L$, so that $L=\IP(V)$
for $V=\IC^6/U$. By assumption, the cubic polynomial $f\in S^3\IC^6$ vanishes
on $L$ and hence is contained in the kernel of $S^3\IC^6\to S^3V$. Its 
leading term is a polynomial $\bar f\in U\tensor S^2V=\Hom(U^*,S^2V)$. That $Y$
is smooth along $L$ is equivalent to saying that the four quadrics in the
image of $\bar f:U^*\to S^2V$ must not have a common zero on $L$. Hence $\bar f$
has at least rank $2$. On the other hand, if $L$ is the line of singularities of 
a non-normal surface $Y\cap \IP(W)$, then $\bar f$ has at most rank $2$, and $W^*
\subset \IC^6{}^*$ is determined as the preimage of $\Kern(\bar f)$ under the
projection $\IC^6{}^*\to U^*$. In particular, every line $L\subset Y$ is the
singular locus of at most one non-normal integral surface of the form $S=Y\cap\IP(W)$.
As the space of lines on a smooth cubic fourfold is four-dimensional, the assertion
follows.
\end{proof}

Since non-normal surfaces form a stratum of codimension $6$ in $\IP(S^3\IC^4)$,
the 'non-normal' locus in $\IG$ is in fact only $2$-dimensional for a generic 
fourfold $Y$.

\begin{proposition}\label{prop:dimensionboundsimpleelliptic} --- 
Let $Y$ be a smooth cubic fourfold not containing a plane.
Then the closure of the set of points $[\IP(W)]\in \IG$ such that $S=\IP(W)\cap Y$
has a simple-elliptic singularity is at most $4$-dimensional.
\end{proposition}

\begin{proof} Let $p\in Y=\{f=0\}$ be a point. Any 3-space $\IP(W)$ with the property
that $S=Y\cap \IP(W)$ is a cone with vertex $p$ must be contained in the
tangent space to $Y$ at $p$. Then one may choose coordinates $x_0,\ldots,x_5$ in 
a way that $x_0,\ldots,x_4$ vanish at $p$, that $x_0=0$ defines the tangent space
and that $f$ takes the form $f=x_5^2x_0+x_5q(x_1,\ldots,x_4)+c(x_0,\ldots,x_4)$
for a quadric polynomial $q$ and a cubic polynomial $c$. If $q$ vanishes identically,
we may choose a line $L$ in $\{x_0=0=c\}\subset\IP^4$. As the plane spanned by
$L$ and $p$ would be contained in $Y$ this case is excluded. A 3-space through
$p$ intersects $Y$ in a cone if and only if it is the span of $p$ and a plane 
in the quadric surface $\{x_0=0=q\}$. Clearly, for any point $p\in Y$ 
there are at most two such planes. Thus the family of such 3-spaces is at most 
$4$-dimensional. 
\end{proof}

Again, the expected dimension of the 'simple-elliptic' locus is much smaller.
We may restate the argument in a coordinate free form as follows: Let $f\in S^3\IC^6$ 
denote the cubic polynomial that defines a smooth fourfold $Y\subset\IP^5$ as 
before. The restriction to $Y$ of the Jacobi map $Jf:\ko_{Y}(-2)\to \ko_{Y}^6$
takes values in $\Omega_{\IP^5}(1)|_Y$. Since $Y$ is smooth, this map vanishes
nowhere, giving rise to a short exact sequence 
$0\to \Omega_Y(1)\to \kf\to \ko_Y(1)\to 0$
with $\kf=\ko_Y^6/\ko_Y(-2)$. By construction, the image of $f$ under the 
canonical map $S^3\IC^6\to H^0(Y,S^3\kf)$ takes values in the subbundle 
$\kf\cdot S^2(\Omega_Y(1))$ with leading term $\tilde f\in 
H^0(Y,S^2(\Omega_Y(1))\tensor \ko_Y(1))=\Hom_Y(\ko_Y(-3),S^2\Omega_Y)$.
Considering $\tilde f$ considered as a symmetric map
$\ko_Y(-3)\tensor \Omega_Y^*\to \Omega_Y$ we may ask for the locus where its 
rank is $\leq 2$. Standard intersection theoretic methods \cite{HarrisTu} allow 
to calculate the expected cycle class as $35 h^3$, where $h$ is the class of a 
hyperplane section in $Y$. This implies:

\begin{corollary} \label{cor:existencesimpleelliptic} --- 
Let $Y$ be a smooth cubic fourfold not containing a plane.
Then there is a 3-space $\IP(W)\subset\IP^5$ such that $Y\cap \IP(W)$ has a
simple-elliptic singularity. 
\end{corollary}

\subsection{The divisor $D\subset Z'$}

A closed point $[C]$ in $D\subset Z'$ corresponds to a family of non-CM curves
on a surface $S=\IP(W)\cap Y$ for some three-dimensional linear subspace $\IP(W)
\subset\IP^5$. In fact, such a family is obtained by intersecting $S$ with
all planes in $\IP(W)$ through a fixed singular point $p\in S$ (and adding the
unique non-reduced structure at $p$). 

On the other hand, if $p\in Y$ is any point, a three-dimensional linear space $\IP(W)$
through $p$ intersects $Y$ in such a way that $p$ becomes a singular point of 
$S=\IP(W)\cap Y$ if and only if $\IP(W)$ is contained in the projective tangent 
space of $Y$ at $p$. This defines a bijective morphism $j:\IP(T_Y)\to D\subset Z'$. 
In fact: 

\begin{proposition} \label{prop:DonTPY} --- Let $\pi:\IP(T_Y)\to Y$ denote the 
projectivisation of the tangent bundle of $Y$. The morphism $j:\IP(T_Y)\to D$ is 
an isomorphism, and $j^*\ko_{Z'}(D)=\ko_\pi(-1)$.
\end{proposition}

\begin{proof} Let $0\to U\to \pi^*T_Y\to \ko_\pi(1)\to 0$ denote the tautological 
bundle sequence on $\IP(T_Y)$. Starting from the Euler sequence on $\IP^5$ we 
obtain the following pull-back diagram of short exact sequences of sheaves on 
$\IP(T_Y)$. 
$$
\begin{array}{ccccccccc}
0&\to& \pi^*\ko_Y(-1)&\to& \IC^6\tensor\ko_{\IP(T_Y)}&\to& 
\pi^*(T_{\IP^5}|_Y\tensor\ko_Y(-1))&\to& 0\\
&&\Big\|&&\Big\uparrow&&\Big\uparrow\\
0&\to& \pi^*\ko_Y(-1)&\to& V'&\to& \pi^*(T_Y\tensor\ko_Y(-1))&\to&0\\
&&\Big\|&&\Big\uparrow&&\Big\uparrow\\
0&\to& \pi^*\ko_Y(-1)&\to& V&\to & U\tensor\pi^*\ko_Y(-1)&\to&0\\
\end{array}
$$
The bundle inclusions $\pi^*\ko_Y(-1)\subset V\subset\IC^6\tensor\ko_{\IP(T_Y)}$ 
induce a closed immersion $u:\IP(T_Y)\to \IP(\kw)$ with $V^*=u^*a^*\kw$ and $u^*\ko_a(1)
=\pi^*\ko_Y(1)$. Moreover, the composite map $\ko_{\IP(T_Y)}\xra{\;f\;}S^3
\ko_{\IP(T_Y)}^6\to u^*a^*S^3\kw$ takes values in the subbundle 
$u^*\kn$ (cf. \eqref{eq:N_appears}), inducing a bundle monomorphism 
$$u^*(\ko_a(1)\tensor a^*\det(\kw)^{-1})\to u^*(\kn')$$ 
and hence a morphism $v:\IP(T_Y)\to \IP(\kn')$ with $\sigma\circ v=u$ and 
\begin{equation}\label{eq:O_sigma_minus_one}
v^*\ko_\sigma(-1)=u^*(\ko_a(1)\tensor a^*\det(\kw)^{-1})
=\pi^*\ko_Y(1)\tensor (a\circ u)^*\det(\kw)^{-1}.
\end{equation}
Adding $u$ and $v$ to diagram \eqref{eq:BigDiagram} we get
\begin{equation}\label{eq:BigDiagram2} 
\begin{array}{c}
\xymatrix{
&\IP(\kn')\ar@{^{(}->}[r]^{j}\ar[d]_{\sigma}&\IH\ar[d]_{\tilde \sigma}\\
\IP(T_Y)\ar[r]_{u}\ar[d]_{\pi}\ar[ur]^{v}&\IP(\kw)\ar[rd]_{a}\ar@{^{(}->}[r]^{i}\ar[d]_{q}&
\IX\ar[d]_{b}\\
Y\ar@{^{(}->}[r]&\IP^5&\IG
}
\end{array}
\end{equation}
Since $(a\circ u)^*\det(\kw)^{-1}=\det(V)=\pi^*\ko_Y(-1)^4\tensor \det(U)$
we may simplify this as follows:
\begin{equation}\label{eq:=_sigma_minus_one_bis}
v^*\ko_\sigma(-1)\isom\pi^*(\det(T_Y)\tensor \ko_Y(-3))\tensor\ko_\pi(-1)\isom\ko_\pi(-1)
\end{equation}
Since $u$ is a closed immersion, so is $v$. By construction, the image of $v$ is 
contained in $D$. This shows that $\IP(T_Y)\isom D_{\red}$. But the pull-back of 
the normal bundle $\ko_\IH(\JJ)|_{\JJ}=\ko_{\sigma}(-1)$ to  $\IP(T_Y)$ equals 
$\ko_\pi(-1)$ according to equation \eqref{eq:=_sigma_minus_one_bis} and hence
is not a power of any other line bundle. This implies that $\IP(T_Y)$ indeed is 
isomorphic to the scheme-theoretic intersection $D=Z'\cap \JJ$ and that 
$\ko_{Z'}(D)|_D=\ko_\pi(-1)$ with respect to the identification $D=\IP(T_Y)$.
\end{proof}

\begin{corollary}\label{cor:SmoothAlongD} --- $Z'$ is smooth along $D$. 
\end{corollary}

\begin{proof} Since $D$ is smooth and a complete intersection in $Z'$, the 
ambient space $Z'$ must be smooth along $D$ as well.  
\end{proof}

\subsection{Smoothness and Irreducibility}
\label{sec:SmoothnessAndIrreducibility}

Let $Y=\{f=0\}\subset\IP^5$ be a smooth cubic hypersurface that does not 
contain a plane. In this section we prove that $\Hilb^{gtc}(Y)$ is smooth
and irreducible. Due to the $\IP^2$-bundle map $a:\Hilb^{gtc}(Y)\to Z'$ 
both assertions are equivalent to the analogous statement about $Z'$.

\begin{theorem} \label{thm:Smoothness} --- 
$\Hilb^{gtc}(Y)$ is smooth of dimension 10. 
\end{theorem}

\begin{proof} 1. Since $\Hilb^{gtc}(Y)$ is the zero locus of a section in a vector
bundle of rank $10$ on a $20$-dimensional smooth variety $\IH_0=\Hilb^{gtc}(\IP^5)$,
every irreducible component of $\Hilb^{gtc}(Y)$ has dimension $\geq 10$. In 
order to proof smoothness, it suffices to show that all Zariski tangent 
spaces are $10$-dimensional. 

Due to the existence of a $\IP^2$-fibre bundle map $a:\Hilb^{gtc}(Y)\to Z'$,
the Hilbert scheme is smooth at a point $[C]$ if and only if $Z'$ is smooth
at $a([C])$, or equivalently, if $\Hilb^{gtc}(Y)$ is smooth at some point
of the fibre $a^{-1}(a([C]))$. And due to Corollary \ref{cor:SmoothAlongD} which
takes care of the non-CM-locus, it suffices to consider aCM-curves, for which 
there is a functorial interpretation of tangent space:
$T_{[C]}\Hilb^{gtc}(Y)\isom \Hom(I_{C/Y},\ko_C)$.

Thus it remains to prove that $\hom(I_{C/Y},\ko_C)=10$ for any generalised twisted 
cubic $C\subset Y$ of aCM-type whose isomorphism type is generic within the
family $a^{-1}(a([C]))$.

2. Given an aCM-curve $C\subset Y$ we may choose coordinates $x_0,\ldots,x_5$ 
in such a way that the ideal sheaf $I_{C/\IP^5}$ is defined by the linear forms 
$x_4$ and $x_5$ and the quadratic minors of a $3\times 2$-matrix $A_0$ with 
linear entries in the coordinates $x_0,\ldots,x_3$. The surface $S=Y\cap \{x_4=x_5=0\}$
is cut out by a cubic polynomial $g\in \IC[x_0,x_1,x_2,x_3]$. There are 
quadratic polynomials $q_4$ and $q_5$ such that $f=g+x_4q_4+x_5q_5$ and linear
forms $\ell_0$, $\ell_1$, $\ell_2$ in $x_0,\ldots,x_3$ such that 
$$g=\det(A)\quad\quad\text{ for }\quad\quad A=\left(A_0\;\Big|\; 
\begin{smallmatrix}\ell_0\\\ell_1\\\ell_2\end{smallmatrix}\right).$$
The ideal sheaf $I_{C/\IP^5}$ has a presentation 
$$\ko_{\IP^5}(-3)^2\oplus \ko_{\IP^5}(-3)^6\oplus\ko_{\IP^5}(-2)
\xra{\;M\;}\ko_{\IP^5}(-2)^3\oplus\ko_{\IP^5}(-1)^2\lra I_{C/\IP^5}\lra 0,$$
with
$$M=\left(\begin{array}{c|c|c}A_0&*&0\\\hline 0&*&*\end{array}\right),$$
where the entries denoted by $*$ give the tautological relations between the quadrics 
and the linear forms defining $I_{C/\IP^5}$. They vanish identically when 
restricted to $C$. Therefore, $\khom(I_{C/\IP^5},\ko_C)=F\oplus \ko_C(1)^2$ 
with $F=\Kern(\ko_C(2)^3\xra{\;A_0^t\;}\ko_C(3)^2)$. Since $Y$ is smooth along 
$C$, the natural homomorphism $\varphi:\khom(I_{C/\IP^5},\ko_C) \to 
N_{Y/\IP^5}|_C=\ko_C(3)$ is surjective, and $\Kern(\varphi)=
\khom(I_{C/Y},\ko_C)$. The homomorphism $\varphi$ can be lifted to 
$\ko_C(2)^3\oplus \ko_C(3)^2$ in such a way that there is an exact sequence
\begin{equation}\label{eq:B-sequence}
0\lra \khom(I_{C/Y},\ko_C)\lra \ko_C(2)^3\oplus \ko_C(1)^2\xra{\; B\;}\ko_C(3)^3
\end{equation}
with 
$$B=\left(\begin{array}{ccc|cc}
&A_0^t&&0&0\\\hline
\ell_0&\ell_1&\ell_2&q_4&q_5
\end{array}\right)$$
Note that $\varphi|_F$ vanishes at a point of $C$ if and only if 
the surface $S$ is singular at this point. We will now analyse $B$
for the four reduced types of aCM-curves. In the first three cases, the curve $C$
is in fact locally a complete intersection, and $N_{C/Y}=\khom(I_{C/Y},\ko_C)$ is 
locally free of rank $3$. 

3. Assume that $C$ is a smooth twisted cubic. For an appropriate choice of 
coordinates we have 
$A_0^t=\left(\begin{smallmatrix}x_0&x_1&x_2\\x_1&x_2&x_3\end{smallmatrix}\right)$,
and we parameterise the curve by 
$$\iota:\IP^1\to C,\quad[s:t]\to [s^3:s^2t:st^2:t^3:0:0].$$ 
Then $\iota^*A_0^t=\left(\begin{smallmatrix}s\\t\end{smallmatrix}\right)\cdot
\left(\begin{smallmatrix}s^2&st&t^2\end{smallmatrix}\right)$ has kernel $\iota^*F=
\ko_{\IP^1}(5)^2$, and
$$\khom(I_{C/Y},\ko_C)\isom \Kern\Big(B':\ko_{\IP^1}(5)^2\oplus\ko_{\IP^1}(3)^2
\to \ko_{\IP^1}(9)\Big)$$
with $B'=(\begin{matrix} t\ell_0-s\ell_1& t\ell_1-s\ell_0&q_4&q_5\end{matrix})$.
The kernel of $B'$ has rank $3$ and degree $7$. Writing it in the form $\ko_{\IP^1}(a)
\oplus\ko_{\IP^1}(b)\oplus\ko_{\IP^1}(c)$ with $5\geq a\geq b\geq c$, it follows 
that either $b\leq 3$ (and hence $c\geq -1$) or $a\geq b\geq 4$. In the first case 
$h^1(N_{C/Y})=0$ and $h^0(N_{C/Y})=10$, as desired. In the second case, we must 
have $\ko_{\IP^1}(5)^2\subset N_{C/Y}$, since the kernel is saturated. But this 
implies that $S$ is singular along $C$, which is impossible. Hence $\Hilb^{gtc}(Y)$
is smooth at any point $[C]$ whose corresponding curve $C$ is smooth.

4. Assume that $C$ is the union of a line $L$ and a quadric $Q$. We may take
$A_0^t=\left(\begin{smallmatrix}x_0&x_1&x_2\\0&x_2&x_3\end{smallmatrix}\right)$, 
so that $L=\{x_2=x_3=0\}$ and $Q=\{x_0=x_1x_3-x_2^2=0\}$. Then 
$A_0^t|_L=\left(\begin{smallmatrix}x_0&x_1&0\\0&0&0\end{smallmatrix}\right)$ 
has kernel $\ko_L(1)\oplus \ko_L(2)$ and 
$$N_{C/Y}|_L=\Kern(B':\ko_L(1)\oplus \ko_L(2)\oplus\ko_L(1)^2\to\ko_L(3))$$ 
with $B'=(\begin{matrix}x_1\ell_0-x_0\ell_1&\ell_2&q_4&q_5\end{matrix})$.
Since $N_{C/X}|_L$ has rank $3$ and degree $2$ and is a subsheaf of 
$\ko_{\IP^1}(2)\oplus\ko_{\IP^1}(1)^3$, it cannot have a direct summand of 
degree $-2$. This implies $h^1(N_{C/Y}|_L)=0$ and hence $h^0(N_{C/Y}|_L)=5$.  
We parameterise the second component of $C$ by $\iota:\IP^1\to Q$, $[s:t]\to 
[0:s^2:st:t^2:0:0]$. The kernel of 
$\iota^*A_0^t=\left(\begin{smallmatrix}0&s^2&st\\0&st&t^2\end{smallmatrix}\right)$
is isomorphic to $\ko_{\IP^1}(4)\oplus\ko_{\IP^1}(3)$, and
$$N_{C/Y}|_Q=\Kern(B':\ko_{\IP^1}(4)\oplus \ko_{\IP^1}(3)\oplus\ko_{\IP^1}(2)^2\to\ko_{\IP^1}(6))$$
with $B'=(\begin{matrix} \ell_0&t\ell_1-s\ell_0&q_4&q_5\end{matrix})$.
The sheaf $\ko_{\IP^1}(4)$ can lie in the kernel only if $\ell_0|_{Q}=0$, i.e.\ 
if $\ell_0$ is a multiple of $x_0$, which is impossible since $x_0$ must not 
divide $\det(A)$. If two copies of $\ko_{\IP^1}(3)$ were contained in the kernel 
they would have to lie in $\ko_{\IP^1}(4)\oplus \ko_{\IP^1}(3)$, and since the 
kernel is saturated, this would imply that $\ko_{\IP^1}(4)$ is contained in the
kernel as well, a case we just excluded. Therefore we have 
$N_{C/Y}|_Q=\ko_{\IP^1}(a)\oplus\ko_{\IP^1}(b)\oplus\ko_{\IP^1}(c)$ 
with $a\geq b\geq c$ and $a\leq 3$ and $b\leq 2$. Since $a+b+c=5$, this implies
$c\geq 0$. Now $N_{C/Y}|_Q$ not only has vanishing $H^1$ but is in fact globally 
generated, so that $H^0(N_{C/Y}|_Q)\to H^0(N_{C/Y}|_{L\cap Q})$ is surjective. 
Hence it follows from the exact sequence
$$0\to H^0(N_{C/Y})\lra H^0(N_{C/Y}|_L)\oplus H^0(N_{C/Y}|_Q)
\lra H^0(N_{C/Y}|_{L\cap Q})$$
that $h^0(N_{C/Y})=5+8-3=10$. 

5. Assume that $C$ is the union of three lines $L_1$, $M$ and $L_2$ that intersect 
in two distinct points $p_1=L_1\cap M$ and $p_2=M\cap L_2$. In appropriate 
coordinates $C$ is defined by the minors of
$A_0^t=\left(\begin{smallmatrix}x_0&x_1&0\\0&x_2&x_3\end{smallmatrix}\right)$,
and $L_1=\{x_0=x_1=0\}$, $M=\{x_0=x_3=0\}$ and $L_2=\{x_2=x_3=0\}$. Then
$A_0^t|_{L_1}=\left(\begin{smallmatrix}0&0&0\\0&x_2&x_3\end{smallmatrix}\right)$
has kernel $F|_{L_1}=\ko_{L_1}(2)\oplus \ko_{L_1}(1)$, so that 
$$N_{C/Y}|_{L_1}=\Kern(B':\ko_{L_1}(2)\oplus\ko_{L_1}(1)^3\to\ko_{L_1}(3))$$ 
with $B'=(\begin{matrix}\ell_0&x_3\ell_1-x_2\ell_2&q_4&q_5\end{matrix})$.
Assume first that $\ell_0|_{L_1}=0$. Then $\ell_0$ must be a linear combination 
of $x_0$ and $x_1$. If it were a multiple of $x_0$, the determinant $\det(A)$ 
would be divisible by $x_0$, contradicting the assumptions on $Y$. Hence 
$\ell_0=\alpha x_0+\beta x_1$ with $\beta\neq 0$. Then for any $\varepsilon\in \IC$ 
the matrix $A_\varepsilon^t=
\left(\begin{smallmatrix}
x_0&x_1&0\\\varepsilon\ell_0&x_2+\varepsilon \ell_1&x_3+\varepsilon\ell_2
\end{smallmatrix}\right)$ defines a curve $C_\varepsilon$ in the $\IP^2$-family 
of $C$, which for generic choice of $\varepsilon$ is the union of a quadric and 
a line. Hence the isomorphism type of $C$ is not generic in the family, and we 
need not further consider this case. If on the other hand $\ell_0|_{L_1}\neq 0$,
then the maximal degree of a direct summand of in the kernel of $B'$ is $1$, so 
that $N_{C/Y}|_{L_1}$ is isomorphic to $\ko_{L_1}(1)^2\oplus \ko_{L_1}$, has 
exactly $5$ global sections and is even globally generated. By symmetry, the same 
is true for $L_2$. 

Similarly, 
$A_0^t|_M=\left(\begin{smallmatrix}0&x_1&0\\0&x_2&0\end{smallmatrix}\right)$
has kernel $F|_M=\ko_M(2)^2$, and 
$$N_{C/Y}|_M=\Kern(B':\ko_M(2)^2\oplus\ko_M(1)^2\to\ko_M(3))$$
with $B'=(\begin{matrix}\ell_0&\ell_2&q_4&q_5\end{matrix})$. Hence $N_{C/Y}|_M$
has degree $3$, and any direct summand has degree $\leq 2$. The only possibility 
for $N_{C/Y}$ not to be globally generated is $N_{C/Y}|_M=
\ko_M(2)^2\oplus\ko_M(-1)$, but even then it has vanishing $H^1$ and hence $h^0=6$. 
Since the restrictions of $N_{C/Y}$ to the lines $L_1$ and $L_2$ are globally generated,
we conclude as in the previous step that the map
$$H^0(N_{C/Y}|_{L_1})\oplus H^0(N_{C/Y}|_M)\oplus H^0(N_{C/Y}|_{L_2})\lra 
H^0(N_{C/Y}|_{p_1})\oplus H^0(N_{C/Y}|_{p_2})$$
is surjective, and that $h^0(N_{C/Y})=5+6+5-3-3=10$.

6. Assume that $C$ is the union of three collinear lines $L_1$, $L_2$ and $L_3$ 
that meet in a point $p$ but are not coplanar. We may take 
$A_0^t=\left(\begin{smallmatrix}x_0&0&-x_2\\0&-x_1&x_2\end{smallmatrix}\right)$
and index the lines so that $x_i$ and $x_3$ are the only non-zero coordinates on $L_i$.
In particular, every column of $A_0^t$ vanishes on two of the lines identically.
We obtain $F=\bigoplus_{i=0}^2 F_i$ with $F_i=\Kern(\ko_C(2)\xra{\;x_i\;}\ko_C(3)
\isom \ko_{L_{i+1}}(1)\oplus\ko_{L_{i+2}}(1)$ with indices taken $\text{mod } 3$,
and need to analyse the exact sequences of the form
\begin{equation}\label{eq:auxseq1}
0\lra N\lra \bigoplus_i \ko_{L_i}(1)^2\oplus\ko_C(1)^2\lra\ko_C(3)\to 0.
\end{equation}
At most one line is contained in the singular locus of $S$. Should this be the
case we may renumber the coordinates so that that line is $L_0$. In any case,
we may restrict sequence \eqref{eq:auxseq1} to $L_0$ and divide out the 
zero-dimensional torsion. We obtain a commutative diagram or purely $1$-dimensional
sheaves with exact columns and rows:
$$
\begin{array}{ccccccccc}
&&0&&0&&0\\
&&\big\downarrow&&\big\downarrow&&\big\downarrow\\
0&\to&N'&\to&\bigoplus_{i=1,2}\big(\ko_{L_i}(1)^2\oplus\ko_{L_i}^2\big)&\to
&\bigoplus_{i=1,2}\ko_{L_i}(2)&\to&0\\
&&\big\downarrow&&\big\downarrow&&\big\downarrow\\
0&\to& N&\to& \bigoplus_{i=0}^2 \ko_{L_i}(1)^2\oplus\ko_C(1)^2&\to&\ko_C(3)&\to& 0\\
&&\big\downarrow&&\big\downarrow&&\big\downarrow\\
0&\to&N''&\to&\ko_{L_0}(1)^2\oplus\ko_{L_0}(1)^2&\to&\ko_{L_0}(3)&\to&0\\
&&\big\downarrow&&\big\downarrow&&\big\downarrow\\
&&0&&0&&0\\
\end{array}
$$
Now $N'=N'_1\oplus N_2'$ where each summand 
$N_i'=\Kern(\ko_{L_i}(1)^2\oplus\ko_{L_i}^2\to \ko_{L_i}(2))$
is a vector bundle of rank $3$ and degree $0$ on $L_i$. Since $S$ is not singular 
along $L_i$ for $i=1,2$, the two summands $\ko_{L_i}(1)$ cannot both be contained 
in $N'$. Necessarily, we have $N'_i\isom \ko_{L_i}(a)\oplus\ko_{L_i}(b)\oplus
\ko_{L_i}(c)$ with $(a,b,c)=(1,0,-1)$, $(0,0,0)$. In any case, $N'$ 
has vanishing $H^1$ and $6$ global sections. On the other hand, $N''$ is locally 
free on $L_0$ of rank $3$ and degree $1$. Admissible decompositions $N''=\ko_{L_0}(a)
\oplus \ko_{L_0}(b)\oplus \ko_{L_0}(c)$ are $(a,b,c)=(1,1,-1)$ and $(1,0,0)$.
In any case, $H^1(N'')=0$ and $h^0(N'')=4$. It follows that $h^0(N)=h^0(N')+h^0(N'')=10$.

7. Assume that $C$ is the first infinitesimal neighbourhood of a line in $\IP^3$, 
defined by, say, $A_0^t=\left(\begin{smallmatrix}x_0&x_1&0\\0&x_0&x_1
\end{smallmatrix}\right)$. 
We will show that the corresponding $\IP^2$-family contains a non-reduced curve, 
so that this case is reduced to those treated before. The curve $C$ necessarily 
forms the singular locus of $S$, and $S$ must be one of the four types of 
non-normal surfaces. In each case there is only one determinantal representation 
up to equivalence and coordinate change, namely
$$A=\left(\begin{smallmatrix} x_0&0&x_2\\x_1&x_0&0\\0&x_1&x_3\end{smallmatrix}\right), 
\left(\begin{smallmatrix} x_0&0&x_1\\x_1&x_0&x_2\\0&x_1&x_3\end{smallmatrix}\right), 
\left(\begin{smallmatrix} x_0&0&x_1\\x_1&x_0&x_2\\0&x_1&x_0\end{smallmatrix}\right), 
\text{ and }
\left(\begin{smallmatrix} x_0&0&x_2\\x_1&x_0&0\\0&x_1&x_0\end{smallmatrix}\right).
$$
A reduced curve in the corresponding $\IP^2$-family is provided for example by the 
matrices
$$A_0'=\left(\begin{smallmatrix}x_0&x_2\\x_1&0\\0&x_3\end{smallmatrix}\right),
\left(\begin{smallmatrix}x_0&x_1\\x_1&x_2\\0&x_3\end{smallmatrix}\right),
\left(\begin{smallmatrix}x_0&x_1\\x_0+x_1&x_2\\x_1&x_0\end{smallmatrix}\right),
\text{ and }
\left(\begin{smallmatrix}x_0&x_2\\x_0+x_1&0\\x_1&x_0\end{smallmatrix}\right),
$$
respectively. 

8. The remaining three types of non-reduced aCM-curves (corresponding to 
matrices $A^{(5)}$, $A^{(6)}$ and $A^{(7)}$ in the enumeration of Section 
\ref{sec:HilbertSchemes}) are each the union of two lines $L$ and $M$, of which 
one, say $L$, has a double structure. As we have already shown that any
$\IP^2$-family containing the most degenerate type also contains a non-reduced 
curve, it suffices to show that there is no $\IP^2$-family parameterising only 
non-reduced curves with two components. Assume that $A\in W^{3\times 3}$ defines
such a family. The corresponding bundle homomorphism is the 
composite map
$$\Omega_{\IP^2}(1)\lra \ko_{\IP^2}^3\xra{\;\; A\;\;}\ko_{\IP^2}^3\tensor W.$$
We form $\Lambda^2\Omega_{\IP^2}(2)\isom\ko_{\IP^2}(-1)\to 
\Lambda^2(\ko_{\IP^2}^3)\tensor S^2W$
and obtain the associated family of nets of quadrics 
$\ko_{\IP^2}(-1)^3\to
\ko_{\IP^2}\tensor S^2W$. 
To each parameter $\lambda\in\IP^2$ in the family there are associated 
subspaces $B_\lambda\subset U_\lambda\subset W$, where $B_\lambda$ defines the 
plane spanned by the lines $L_\lambda$ and $M_\lambda$, and $U_\lambda$ defines 
the line $L_\lambda$. Let $\kb\subset \ku\subset \ko_{\IP^2}\tensor_\IC W$ 
denote the corresponding vector bundles. Then there are inclusions
$$\kb\cdot\ku\subset \ko_{\IP^2}(-1)^3\subset \ko_{\IP^2}\tensor S^2W.$$
But such a configuration of vector bundles is impossible: Both inclusions
$\kb\subset \ku$ and $\kb\ku\subset\ko_{\IP^2}(-1)^3$ would have to split,
say $\kb=\ko_{\IP^2}(a)$, $\ku=\ko_{\IP^2}(a)\oplus\ko_{\IP^2}(b)$  
and finally $\ko_{\IP^2}(-1)^3\isom \ko_{\IP^2}(2a)\oplus \ko_{\IP^2}(a+b)
\oplus\ko_{\IP^2}(c)$, and the latter isomorphism is clearly impossible.
\end{proof}

\begin{theorem} \label{thm:Irreducibility} --- 
$Z'$ is an $8$-dimensional smooth irreducible projective variety. 
\end{theorem}

\begin{proof} Due to the existence of the $\IP^2$-fibration $\Hilb^{gtc}(Y)\to Z'$,
the smoothness of $\Hilb^{gtc}(Y)$ implies that $Z'$ is smooth as well and of 
dimension $8$. The morphism $Z'\to \IG$ is finite over the open subset of 
ADE-surfaces, and has fibre dimension $\leq 1$ resp. $\leq 2$ over the strata of
simple-elliptic and non-normal surfaces, resp., due to Corollary 
\ref{cor:dimensionestimate}.
By Proposition \ref{prop:dimensionboundsimpleelliptic} and Proposition 
\ref{prop:dimensionboundnonnormal}, simple-elliptic and non-normal surfaces form
strata in $\IG$ of dimension $\leq 4$. It follows that every irreducible
component of $Z'$ must dominate $\IG$. The stratum of simple-elliptic surfaces
in $\IG$ is non-empty by Corollary \ref{cor:existencesimpleelliptic}. Since
$\Hilb^{gtc}(S)$ is connected for a simple-elliptic surface,
$Z'$ must be connected as well. Being smooth, $Z'$ is irreducible.
\end{proof}

Again, due to the existence of the $\IP^2$-fibre bundle map $\Hilb^{gtc}(Y)\to Z'$, 
this theorem is equivalent to Theorem \ref{thm:SmoothAndIrreducible}.

\subsection{Symplecticity}

We continue to assume that $Y\subset\IP^5$ is a smooth hypersurface that does
not contain a plane.

De Jong and Starr \cite{deJongStarr} showed that any smooth projective model
of the coarse moduli space associated to the stack of rational curves of degree
$d$ on a very general cubic fourfold carries a natural 2-form $\omega_d$. 
In our context, $\omega_3$ can be defined directly as follows:
Let $\Omega=\sum_{i=0}^5(-1)^i x_i dx_0\wedge\ldots \widehat{dx_i}\ldots \wedge dx_5$. 
An equation $f$ for $Y$ determines a generator $\alpha\in H^{3,1}(Y)$ 
as the image of $[\Omega/f^2]$ under Griffiths's residue isomorphism 
$$\text{Res}:H^5(\IP^5\setminus Y,\IC)\to H^4_{\text{prim}}(Y).$$ 
The cycle $[\kc]\in H_{22}(\Hilb^{gtc}(Y)\times Y;\IZ)$
of the universal curve 
$\kc\subset \Hilb^{gtc}(Y)\times Y$ defines a correspondence 
$$[\kc]_*:H^{4}(Y,\IC)\to H^2(\Hilb^{gtc}(Y),\IC)$$ 
via $[\kc]_*(u)=\PD^{-1}\pr_{1*}(\pr_2^*(u)\cap [Z])$,
where $\pr_1$ and $\pr_2$ denote the projections
from $\Hilb^{gtc}(Y)\times Y$ to its factors.
Since the homology class $[\kc]$ is algebraic, the map $[\kc]_*$ is of Hodge type 
$(-1,-1)$ and maps $H^{3,1}(Y)\isom\IC$ to $H^{2,0}(\Hilb^{gtc}(Y))$. Let the 
two-form $\omega_3$ be the image of $\alpha\in H^{3,1}(Y)$. More importantly, 
de Jong and Starr showed that the value of $\omega_3$ on the tangent space 
$T_{[C]}\Hilb^{gtc}(Y)=H^0(C,N_{C/Y})$ at a {\sl smooth} rational curve $C\subset Y$ 
has the following geometric interpretation:

There is a short exact sequence of normal bundles
\begin{equation}\label{eq:NormalbundleSequence}
0\to N_{C/Y}\to N_{C/\IP^5}\to N_{Y/\IP^5}|_C\to 0.
\end{equation}
To simplify the notation let $A:=N_{C/Y}$, $N:=N_{C/\IP^5}$ and $F:=N_{Y/\IP^5}$.
The fact, 
that $Y$ is a {\sl cubic} contributes the relation
\begin{equation}\frac{\det A}{F}\isom \frac{\det N}{F^2}\isom 
\frac{\omega_C}{\omega_{\IP^5}\tensor F^2}\isom \omega_C.\end{equation}
Taking the third exterior power of \eqref{eq:NormalbundleSequence} and dividing 
by $F$ one obtains a short exact sequence
\begin{equation}0\to \frac{\det A}{F}\to \frac{\Lambda^3N}{F} \to \Lambda^2 A\to 0,\end{equation}
whose boundary operator defines a skew-symmetric pairing
\begin{equation}\delta:\Lambda^2 H^0(A)\to H^0(C,\Lambda^2A)\to H^1(C,\det(A)\tensor F^*)
=H^1(C,\omega_C)\isom\IC.\end{equation}
By Theorem 5.1 in \cite{deJongStarr}, one has $\omega_3(u,v)=\delta(u\wedge v)$
for any two tangent vectors $u,v\in H^0(C,N_{C/Y})$, up to an irrelevant constant factor. 
By a rather involved calculation de Jong and Starr show that $\omega_3$ generically 
has rank $8$. We will need the following minimally sharper result:

\begin{proposition} \label{prop:SymplecticityInSmoothCurves} --- 
$\omega_3$ has rank $8$ at $[C]\in \Hilb^{gtc}(Y)$ whenever $C$ is smooth.
\end{proposition}

\begin{proof} Consider the second exterior power of \eqref{eq:NormalbundleSequence} 
and divide again by $F$:
\begin{equation}\label{eq:NormalbundleSequence2}
0\to \frac{\Lambda^2 A}{F}\to \frac{\Lambda^2 N}{F}\to A\to 0.
\end{equation}
Note that $\Lambda^2A/F\isom A^*\tensor \det A/F\isom A^*\tensor \omega_C$. The 
associated boundary operator defines a map
\begin{equation}
\delta':H^0(C,A)\lra H^1(C,\Lambda^2 A\tensor F^*)\isom H^0(C,A)^*.
\end{equation}
The commutative diagram
\begin{equation*}
\begin{array}{ccccc}
H^0(A)\tensor H^0(A)&\to&H^0(A)\tensor H^1(\Lambda^2A/F)&\to
&H^0(A)\tensor H^1(A^*\tensor \omega_C)\\
\Big\downarrow&&\Big\downarrow&&\Big\downarrow\\
H^0(\Lambda^2A)&\to&H^1(\det A /F)&\isom &H^1(\omega_C)
\end{array}
\end{equation*}
shows that $\delta'$ is the associated linear map of the pairing $\delta$. 

Though it is less clear from $\delta'$ that the pairing on $H^0(A)$ is 
skew symmetric, it makes it easier to compute the radical of $\omega_3$ at $[C]$, 
which is simply the kernel of $\delta'$ and hence the cokernel of the injective 
homomorphism $\gamma:H^0(C, \Lambda^2 A\tensor F^*)\to 
H^0(C,\Lambda^2 N \tensor F^*)$ induced by \eqref{eq:NormalbundleSequence2}. 
Using an identification $C\isom \IP^1$ we have isomorphisms $F\isom 
\ko_{\IP^1}(9)$ and $N\isom \ko_{\IP^1}(5)^2\oplus\ko_{\IP^1}(3)^2$. 
The 
bundle $\Lambda^2N\tensor F^*\isom\ko_{\IP^1}(1)\oplus \ko_{\IP^1}(-1)^4\oplus 
\ko_{\IP^1}(-3)$ has exactly two sections. If we write $A=\ko_{\IP^1}(a)\oplus 
\ko_{\IP^1}(b)\oplus\ko_{\IP^1}(c)$ with $a\geq b\geq c$ then $a+b+c=\deg(A)=7$,
and we know from step 3 in the proof of Theorem \ref{thm:Smoothness} that
$c\geq -1$ and $a+b\leq 8$. Thus the maximal degree of a direct
summand of $\Lambda^2A/F$ is $a+b-9\leq -1$. This shows $h^0(\Lambda^2A/F)=0$
and $\dim \rad\omega_3([C])=\dim\Cok(\gamma)=h^0(\Lambda^2N/F)=2$.
\end{proof}

\begin{theorem}--- Let $a:\Hilb^{gtc}(Y)\to Z'$ be the $\IP^2$-fibration constructed before.
\begin{enumerate}
\item There is a unique form $\omega'\in H^0(Z',\Omega^2_{Z'})$ 
such that $a^*\omega'=\omega_3$.
\item $\omega'$ is non-degenerate on $Z'\setminus D$.
\item $K_{Z'}=mD$ for some $m>0$.
\end{enumerate}
\end{theorem}

\begin{proof} 1. From the exact sequence
$0\to a^*\Omega_{Z'}\to \Omega_{M_3 }\to \Omega_{M_3 /Z'}\to 0$
one gets a filtration by locally free subsheaves
$0\subset a^*\Omega^2_{Z'}\subset U\subset \Omega^2_{M_3 }$ with factors
$U/a^*\Omega^2_{Z'}\isom a^*\Omega_{Z'}\tensor\Omega_{M_3 /Z'}$ and 
$\Omega^2_{M_3 }/U\isom \Omega^2_{M_3 /Z'}$.
This in turn yields exact sequences
\begin{equation}
0\lra H^0(M_3 ,U)\lra H^0(M_3 ,\Omega^2_{M_3 })\lra H^0(M_3 ,\Omega^2_{M_3 /Z'})
\end{equation}
and
\begin{equation}0\lra H^0(M_3 ,a^*\Omega_Z^2)\lra H^0(M_3 ,U)\lra 
H^0(M_3 ,a^*\Omega_{Z'}\tensor\Omega_{M_3 /Z'}).\end{equation}
Since neither $\Omega_{\IP^2}$ nor $\Omega_{\IP^2}^2$ have nontrivial sections, 
$a_*\Omega_{M_3 /Z'}$ and $a_*\Omega^2_{M_3 /Z'}$ vanish. It 
follows that $H^0(M_3 ,\Omega_{M_3 /Z'}^2)=H^0(Z',a_*\Omega^2_{M_3 /Z'})=0$ and 
$H^0(M_3 ,a^*\Omega_{Z'}\tensor\Omega_{M_3 /Z'})=H^0(Z',\Omega_{Z'}\tensor 
a_*\Omega_{M_3 /Z'})=0$. We are left with isomorphisms
\begin{equation}H^0(Z',\Omega_{Z'}^2)\isom H^0(M_3 ,a^*\Omega_{Z'}^2)\isom H^0(M_3 ,U)\isom 
H^0(M_3 ,\Omega_{M_3 }^2).\end{equation}
This shows that $\omega_3$ descends to a unique $2$-form $\omega'$ on $Z'$. 

2. It follows from Proposition \ref{prop:SymplecticityInSmoothCurves} that $\omega'$
is non-degenerate at all points $z\in Z'$ 
for which the fibre $a^{-1}(z)$ contains a point corresponding to a smooth rational
curve.
%
By Theorem \ref{thm:ADESurfaceCase}, this is the case
for all points corresponding to fibres with aCM-curves on a surface with at most
ADE-singularities. The dimension argument in the proof of Theorem 
\ref{thm:Irreducibility} shows that the locus of points in $Z'\setminus D$ that
do not satisfy this condition has codimension $\geq 2$. But the degeneracy locus
of a $2$-form is either empty or a divisor. Thus $\omega'$ is indeed non-degenerate
on $Z'\setminus D$. 

3. Since $\omega'$ is non-degenerate on $Z'\setminus D$, its $4$th exterior 
power defines a non-vanishing section in the canonical line bundle of $Z'$ over 
$Z'\setminus D$, showing that 
$K_{Z'}=mD$ for some $m\geq 0$. To see that $m>0$, it suffices to note that $Y$
has no non-trivial holomorphic $2$-form, so that the restriction of $\omega'$ 
to $D=\IP(T_Y)$ must vanish identically. Consequently $\omega'$ must
be degenerate along $D$.
\end{proof}

A calculation of the topological Euler characteristic of the preimage curve in $Z'$ 
of a generic line $L\subset\Grass(\IC^6,4)$ 
shows that $K_{Z'}\sim 3D$. We will not need this explicit number and hence omit 
the calculation. In fact, $m=3$ easily follows {\sl a posteriori} once we have shown 
the existence of a contraction $Z'\to Z$ to a manifold $Z$ that maps $D$ to $Y$. 

\subsection{The extremal contraction}

\begin{theorem}\label{thm:Contraction} --- 
There exists an $8$-dimensional irreducible projective manifold
$Z$ and a morphism $\Phi:Z'\to Z$ with the following properties:
\begin{enumerate}
\item $\Phi$ maps $Z'\setminus D$ isomorphically to $Z\setminus\Phi(D)$.
\item $\Phi|_D$ factors through the projection $\pi:D=\IP(T_Y)\to Y$ and a
closed immersion $j:Y\to Z$.
\item There is a unique holomorphic $2$-form $\omega\in H^0(Z,\Omega_Z^2)$ 
such that $\omega'=\Phi^*\omega$.
\item $\omega$ is symplectic.
\end{enumerate}
\end{theorem}

We will prove the theorem in several steps:

\begin{lemma}--- The line bundle $\ko_{Z'}(D)$ is ample relative to
$s:Z'\to \IG$.
\end{lemma}

\begin{proof} As the statement is relative over the Grassmannian, it suffices to prove 
the analogous statement for the divisor $J\subset H$ relative to the morphism 
$H\to \IP(S^3W^*)$. This is the content of Corollary \ref{cor:relativeAmpleness}.
\end{proof}

Let $\kw$ denote the universal rank 4 bundle on $\IG$. Then $\det(\kw)$ 
is very ample, and its pull-back $B:=s^*\det(\kw)$ to $Z'$ is a nef line bundle. 
The linear system of the line bundle 
$$L:=\ko_{Z'}(D)\tensor B.$$
will produce the contraction $\Phi:Z'\to Z$. It follows from Proposition \ref{prop:DonTPY}
that with respect to the identification $D=\IP(T_Y)$ we have
\begin{equation}\label{eq:NegativityOfD}
\ko(D)|_D=\ko_\pi(-1)\quad\quad\text{and}\quad\quad L|_D\isom \pi^*\ko_Y(1).
\end{equation}

\begin{lemma}\label{lem:ray1} --- $L$ is nef, and all irreducible curves $\Sigma\subset Z'$ with 
$\deg(L|_{\Sigma})=0$ are contained in $D$, and more specifically, in the fibres 
of $\pi:D=\IP(T_Y)\to Y$.
\end{lemma}

\begin{proof} Assume first, that $\Sigma$ is an irreducible curve not contained 
in $D$. Since $D$ is effective, $D.\Sigma\geq 0$. As $B$ is nef, one has 
$\deg(L|_\Sigma)\geq 0$. Moreover, $\deg(L|_\Sigma)>0$ unless $\deg(B|_\Sigma)=0$, 
which is only possible when $\Sigma$ is contained in the fibres of 
$Z'\to \Grass(\IC^6,4)$. But since $D$ is relatively ample over the Grassmannian, 
one would have $D.\Sigma>0$. 

Conversely, if $\Sigma\subset D$, we have $\deg(L|_\Sigma)=
\deg(\ko_Y(1)|_\pi(\Sigma))\geq 0$ by the previous lemma. This number is $>0$ unless 
$\Sigma$ lies in the fibre of $\pi:D\to Y$.
\end{proof}

\begin{lemma}\label{lem:Ampleness}--- 
For all $p,q>0$ the line bundle $L^p\tensor B^q$ is ample.
\end{lemma}
 
\begin{proof} As $B$ is the pull-back of an ample line bundle on $\IG$ and $L$
is ample relative $\IG$, it follows that $L\tensor B^\ell$ is ample for some large
$\ell$. Since both $L$ and $B$ are both nef, $L^{1+m}\tensor B^{\ell+n}$ is ample
for all $m,n\geq 0$ by Kleiman's numerical criterion for ampleness \cite{Kleiman}.
\end{proof}

\begin{lemma}\label{lem:ray2} --- 
The classes $[\Sigma]$ of curves with $\deg(L|_\Sigma)=0$ form a $K_{Z'}$-negative 
extremal ray.
\end{lemma}

\begin{proof} According to the previous lemma, curves with $\deg(L|_\Sigma)=0$ 
are contained in the fibres of a projective bundle $D=\IP(T_Y)\to Y$. Any such 
curve is numerically equivalent to a multiple of a line in any of these fibres. Such classes 
$[\Sigma]$ generate a ray. Moreover, as $\ko_D(D)$ is negative on the fibres of 
$\pi$ by \eqref{eq:NegativityOfD} and $K_{Z'}\sim mD$, the restriction of 
$K_{Z'}$ to this ray is strictly negative.
\end{proof}

Using the Contraction Theorem (\cite{KollarMori} Thm.\ 3.7,  or \cite{Matsuki} Thm.\ 8-3-1) 
we conclude: There is a morphism $Z'\to Z$ with the following properties:
\begin{enumerate}
\item $Z$ is normal and projective, $\Phi$ has connected fibres, and $\Phi_*\ko_{Z'}=\ko_Z$.
\item A curve $\Sigma\subset Z'$ is contracted to a point in $Z'$ if and only if its class
is contained in the extremal ray.
\item There is an ample line bundle $L'$ on $Z$ such that $L\isom \Phi^*L'$.
\end{enumerate}

Let $Y'\subset Z$ denote the image of $D$. By Lemma \ref{lem:ray1} and 
Lemma \ref{lem:ray2}, the morphism $\Phi$ contracts exactly the fibres 
of $\pi:\IP(T_Y)\to Y$. Since the fibres of $\pi$ and of $\Phi$ 
are connected, $\Phi$ induces bijections $Z'\setminus D\to Z\setminus Y'$ 
and $Y\to Y'$. As both $Z'\setminus D$ and $Z\setminus Y'$ are normal, the 
restriction $\Phi: Z'\setminus D\to Z\setminus Y'$ is an isomorphism. 

\begin{lemma}--- For sufficiently large $\ell$ the natural map 
$H^0(Z', L^\ell)\to H^0(D,L^\ell|_D)$ is surjective. 
\end{lemma}

\begin{proof}  By Lemma \ref{lem:Ampleness}, 
$$L^\ell(-D)\tensor \ko(-K_{Z'})=L^\ell(-(m+1)D)=B^{m+1}\tensor 
L^{\ell-m-1}$$ 
is ample for $\ell>m+1$. Hence an application of the Kodaira Vanishing Theorem
gives $H^1(Z', L^\ell(-D))=0$, so that $H^0(Z', L^\ell)\to H^0(D,L^\ell|_D)$ 
is surjective. 
\end{proof}
 
Since $L|_D\isom \pi^* \ko_Y(1)$ it follows from the previous lemma that 
$Y\to Y'$ is an isomorphism. 

\begin{proposition} --- $Z$ is smooth. 
\end{proposition}

\begin{proof} It remains to show that $Z$ is smooth along $Y$. The system of 
ideal sheaves $I_n:=\Phi^{-1}(I^n_{Y/Z})\ko_{Z'}$ and $\ko_{Z'}(-nD)$ are cofinal.
Moreover, there are exact sequences
$$0\lra \ko_D(-nD)\lra \ko_{(n+1)D}\lra \ko_{nD}\lra 0$$
and 
$$0\lra S^nT_Y\lra \Phi_*\ko_{(n+1)D}\lra \Phi_*\ko_{nD}\lra 0,$$
since $\ko_D(-nD)=\ko_\pi(n)$ and thus $\Phi_*\ko_D(-nD)=S^nT_Y$ and 
$R^i\Phi_*\ko_D(-nD)=0$ for all $i>0$. It follows from Grothendieck's version of 
Zariski's Main Theorem (\cite{EGA}, Thm. III.4.1.5.) that the completion of $Z$ 
along $Y$ can be computed by
$$\hat\ko_Z=\underleftarrow{\lim} \Phi_*(\ko_{Z'}/I_n)=
\underleftarrow{\lim} \Phi_*(\ko_{nD})=\hat S(T_Y).$$
This shows that $Z$ is smooth along $Y$.
\end{proof}

\begin{proposition} --- The form $\omega'$ on $Z'$ descends to a symplectic form 
$\omega$ on $Z$.
\end{proposition}

\begin{proof} As $Y\subset Z$ has complex codimension $4$, the pull-back of $\omega'$
via the isomorphism $Z\setminus Y\to Z'\setminus D$ extends uniquely to a holomorphic
$2$-form $\omega$ that is necessarily symplectic since the degeneracy locus of
a $2$-form is either empty or a divisor.
\end{proof}

This finishes the proof of Theorem \ref{thm:Contraction}.

\subsection{Simply connectedness.}

\begin{proposition}--- $Z$ is irreducible holomorphic symplectic, i.e.\ $Z$ is 
simply-connected and $H^0(\omega_Z)=\IC\omega$. In particular,  $Z$ carries a 
Hyperkähler metric.
\end{proposition}

\begin{proof} The first Chern class of $Z$ is trivial. By Beauville's Théorème 1 
in \cite{BeauvilleClasse}, there is a finite étale covering $f:\tilde Z
\to Z$ such that $\tilde Z\isom \prod_i Z_i$, where each factor $Z_i$ is either
irreducible holomorphic symplectic, a torus or a Calabi-Yau manifold. In fact,
since $\tilde Z$ carries a non-degenerate holomorphic $2$-form, factors of
Calabi-Yau type are excluded. As $Y$ is simply-connected, the inclusion $i:Y\to Z$
lifts to an inclusion $a:Y\to \tilde Z$. Let $k$ be an index such that the
projection $a_k:Y\to Z_k$ is not constant. Since $\Pic(Y)=\IZ$, the morphism $a_k$
must be finite. Since $H^0(\Omega_Y^1)=0$, $Z_k$ cannot be a torus. And since 
$H^0(\Omega_Y^2)=0$, the tangent space $T_yY$ of any point $y\in Y$ must map
to an isotropic subspace of $T_{a_k(y)}Z_k$, which requires $\dim(Z_k)\geq 2\dim(Y)=8$.
This shows that there is only one factor in the product decomposition and that
$\tilde Z$ is itself irreducible holomorphic symplectic. Moreover, since $f$ is étale,
we have $H^0(\Omega^i_Z)\subset H^0(\Omega_{\tilde Z}^i)$ and get inequalities
$h^0(\Omega^{2i-1}_Z)\leq h^0(\Omega^{2i-1}_{\tilde Z})=0$ for $i=1,\ldots,4$
and
$1\leq h^0(\Omega^{2i}_Z)\leq h^0(\Omega^{2i}_{\tilde Z})=1$ for $i=0,\ldots,5$.
In particular, $\chi(\ko_Z)=\sum_{i=0}^8(-1)^ih^0(\Omega^i_Z)=5$ and similarly 
$\chi(\ko_{\tilde Z})=5$. On the other hand, it follows from the Hirzebruch-Riemann-Roch
theorem that
$$\chi(\ko_{\tilde Z})=\int_{\tilde Z}\td(T_{\tilde Z})=\int_{\tilde Z}\td(f^*T_{Z})
=\deg(f)\int_Z\td(T_Z)=\deg(f)\chi(\ko_Z).$$
We concluce that $\deg(f)=1$ and that $Z$ is irreducible holomorphic symplectic.
\end{proof}

\subsection{The topological Euler number}

\begin{theorem}--- The topological Euler number of $Z$ is 25650.
\end{theorem}

This number equals the Euler number of the Hilbert scheme $\Hilb^4(K3)$
of $0$-dimen\-sio\-nal subschemes of length $4$ on a K3-surface \cite{Goettsche}. 
This and the fact that the Beauville-Donagi moduli space of lines on $Y$ is
isomorphic to $\Hilb^2$ of a K3-surface if $Y$ is of Pfaffian type makes
it very hard not to believe that $Z$ is isomorphic to some $\Hilb^4(K3)$
for special choices of $Y$ or is at least deformation equivalent to such
a Hilbert scheme. 

For this reason we will not give a detailed proof of the theorem here. Our
method imitates the pioneering calculations of Ellingsrud and Str\o mme
\cite{EllingsrudStromme}. Note first that $e(Z')=e(Z)+e(Y)(e(\IP^3)-1)=
e(Z)+81$ and $e(\Hilb^{gtc}(Y))=e(Z')e(\IP^2)=3e(Z')$. Hence the assertion is
equivalent to $e(\Hilb^{gtc}(Y))=77193$. Now $\Hilb^{gtc}(Y)$ is the zero locus of a regular
section in a certain 10-dimensional tautological vector bundle $A$ on 
$\Hilb^{gtc}(\IP^5)$ (cf.\ Section \ref{sec:HilbertSchemes}). It is 
therefore possible to
explicitly express both the class of $\Hilb^{gtc}(Y)$ and the Chern classes of 
its tangent bundle in terms of tautological classes in the cohomology
ring $H^*(\Hilb^{gtc}(\IP^5),\IQ)$. Two options present themselves for the
calculation: 
 
1. Follow the model of Ellingsrud and Str\o mme and write down a presentation
of the rational cohomology ring of $\Hilb^{gtc}(\IP^5)$ in terms of
generators and relations and calculate using Groebner base techniques.
This is the option we chose. We wrote pages of code first in SINGULAR
and then in SAGE \cite{LehnSorger}. 

2. Take a general linear $\IC^*$ action on $\IP^5$ and determine the
induced local weights at any of the 1950 fixed points for the induced
action on $\Hilb^{gtc}(\IP^5)$. Fortunately there are only nine different 
types of fixed points. The relevant calculations can then
be executed by means of the Bott-formula.



\begin{thebibliography}{mmm}
\bibitem{Artin} M.\ Artin: On isolated rational singularities of surfaces. 
Am.~J.~Math.~88 (1966), pp. 129-136.
%
\bibitem{AuslanderBuchweitz} M.\ Auslander, R.-O.\ Buchweitz:
The homological theory of maximal Cohen-Macaulay approximations. 
Colloque en l'honneur de Pierre Samuel (Orsay, 1987).
Mém. Soc. Math. France (N.S.) No. 38 (1989), 5–37. 
%
\bibitem{BeauvilleClasse} A.\ Beauville: Variétés Kähleriennes dont la première 
classe de Chern est nulle. J.\ Diff.\ Geom.\ 18 (1983) 755–782.
%
\bibitem{Beauville} A.\ Beauville: Moduli of cubic surfaces and Hodge Theory 
(after Allcock, Carlson, Toledo). Géométries à courbure négative ou nulle, 
groupes discrets et rigidités, Séminaires et Congrès 18, 446-467; SMF, 2008.
%
\bibitem{BeauvilleDonagi} A.\ Beauville, R.\ Donagi: La vari\'et\'e des droites 
d'une hypersurface cubique de dimension 4. C.R.~Acad.~Science Paris 301 (1982), 
pp. 703-706.
%
\bibitem{BeauvilleHypersurfaces} A.\ Beauville: Determinantal hypersurfaces.
Michigan Math.\ J.\ 48 (2000), 39-64.
%
\bibitem{Bourbaki} N.\ Bourbaki: Groupes et Alg\`ebres de Lie, Ch. VI. Masson 1981.
%
\bibitem{BruceWall} J.~W.~Bruce, C.~T.~C.~Wall: On the classification of cubic
surfaces. J.\ London Math.\ Soc.\ 19 (1979), 245-256.
%
\bibitem{deJongStarr} J.\ de Jong, J.\ Starr: Cubic fourfolds and spaces of 
rational curves. Ill.\ J.\ Math.\ 48 (2004), 415- 450.
%
\bibitem{Demazure} M.\ Demazure: Surface de Del Pezzo I-IV, in: S\'eminaire sur 
les Singularit\'es des Surfaces. Springer Lecture Notes 777 (1980), pp. 21-69.
%
\bibitem{Dolgachev} I.~Dolgachev: Classical Algebraic Geometry. A Modern View. 
Cambridge UP 2012.
%
\bibitem{Drezet} J.~M.~Drezet: Fibr\'es exceptionnels et vari\'et\'es de modules de 
faisceaux semi-stables sur $\IP_2(\IC)$. J.~Reine Angew.\ Math.\ 380 (1987), 
pp. 14–58.
%
\bibitem{Eisenbud} D.~Eisenbud: Homological algebra on a complete intersection,
with an application to group representations. Transactions of the AMS 260 (1981),
35 - 64.
%
\bibitem{EllingsrudStromme} G.\ Ellingsrud, S.\ A.\ Str\o mme: The number of 
twisted cubic curves on the general quintic threefold. Math.\ Scand.\ 76 (1995), 
pp. 5-34.
%
\bibitem{EllingsrudPieneStromme} G.\ Ellingsrud, R.\ Piene, S.\ A.\ Str\o mme: 
On the variety of nets of quadrics defining twisted curves. In: F.\ Ghione, Ch.\ 
Peskine, E.\ Sernesi: Space curves. Lecture Notes in Math.\ 1266, Springer 1987.
%
\bibitem{EGA} A.\ Grothendieck: \'El\'ements de g\'eom\'etrie alg\'ebrique.
Publ. Math. de l'IHES, 11 (1961),  pp. 5-167.
%
\bibitem{Goettsche} L.\ G\"ottsche: The Betti numbers of the Hilbert scheme 
of points on a smooth projective surface. Math.\ Ann.\ 286 (1990), 193 – 207.
%
\bibitem{HarrisTu} J.~Harris, L.~W.~Tu: On symmetric and skew-symmetric determinantal
varieties. Topology 23 (1984), 71-84.
%
\bibitem{Hartshorne} R.~Hartshorne: Algebraic Geometry. Graduate Texts in 
Mathematics vol.\ 52, Springer 1977.
%
\bibitem{Henderson} A.~Henderson: The twenty-seven lines upon the cubic surface.
Hafner Publishing Company, New York, 1911.
%
\bibitem{Hulek} K.~Hulek: On the classification of stable rank-r vector bundles 
over the projective plane. In: Vector bundles and differential equations 
(Proc.\ Conf.\ Nice, 1979), pp. 113–144, Progr. Math. 7, Birkhäuser, 1980.
%
\bibitem{Humphreys} J.~Humphreys: Reflection Groups and Coxeter Groups. 
Cambridge University Press 1972.
%
\bibitem{Kleiman} S.L.~Kleiman: Toward a numerical theory of ampleness.
Ann.\ of Math.\ 84 (1966), pp. 293-344.
%
\bibitem{Matsuki} K.~Matsuki: Introduction to the Mori program. Springer Universitext.
Springer 2001.
%
\bibitem{KollarMori} J.~Kollár, S.~Mori: Birational geometry of algebraic varieties.
Cambridge Tracts in Math.\ 134, Cambridge University Press, 1998.
%
\bibitem{LehnSorger} M.~Lehn, C.~Sorger: Chow - A SAGE package for computations 
in intersection theory. http://www.math.sciences.univ-nantes.fr/$\sim$ sorger/chow.
%
\bibitem{Looijenga} E.~Looijenga: On the semi-universal deformation of a simple
elliptic hypersurfaces singularity. Part II: the discriminant. Topology 17 (1978), 23-40.
%
\bibitem{Mumford} D.\ Mumford, J.\ Fogarty: Geometric Invariant Theory. 
Springer 1982.
%
\bibitem{Newstead} P.\ Newstead: Lectures on Introduction to moduli problems and
orbit spaces. Tata Institute Lecture Notes, Springer 1978.
%
\bibitem{PieneSchlessinger} R.\ Piene, M.\ Schlessinger: On the Hilbert scheme 
compactification of the space of twisted cubics. Amer.\ J.\ Math.\ 107 (1985), 
no. 4, pp. 761–774.
%
\bibitem{Saito} K.\ Saito: Einfach-elliptische Singularitäten. Invent.\ Math.\
23 (1974), 289-325.
%
\bibitem{Schlaefli} L.\ Schläfli: On the distribution of surfaces of the third order 
into species. Phil.\ Trans.\ Roy.\ Soc. 153 (1864), 193-247.
%
\bibitem{Wall} C.\ T.\ C.\ Wall: Root systems, subsystems and singularities. 
J.~Alg.~Geom.\ 1 (1992), pp. 597-638.
\end{thebibliography}
\end{document}